\documentclass{article}
\title{Aspects of Predicative Algebraic Set Theory III: Sheaves}
\author{Benno van den Berg\footnote{Mathematisch Instituut, Universiteit Utrecht,
PO. Box 80010, 3508 TA Utrecht, the Netherlands. Email address: B.vandenBerg1@uu.nl. Supported by Netherlands Organisation for Scientific Research (NWO project ``The Model Theory of Constructive Proofs'').} \, \& Ieke Moerdijk\footnote{Corresponding author. Institute for Mathematics, Astrophysics and Particle Physics, Radboud University, Heyendaalseweg 135, 6525 AJ Nijmegen, the Netherlands. Email address: i.moerdijk@math.ru.nl.}}
\date{September 29, 2011}
\usepackage{ast2}
\begin{document}

\maketitle

\begin{abstract}
\noindent
This is the third installment in a series of papers on algebraic set theory. In it, we develop a uniform approach to sheaf models of constructive set theories based on ideas from categorical logic. The key notion is that of a ``predicative category with small maps'' which axiomatises the idea of a category of classes and class morphisms, together with a selected class of maps whose fibres are sets (in some axiomatic set theory). The main result of the present paper is that such predicative categories with small maps are stable under internal sheaves. We discuss the sheaf models of constructive set theory this leads to, as well as ideas for future work.\footnote{MCS: 18F20; 03F50; 03E70.}
\end{abstract}

\section{Introduction}

This is the third in a series of papers on algebraic set theory, the aim of which is to develop a categorical semantics for constructive set theories, including predicative ones, based on the notion of a ``predicative category with small maps''.\footnote{Accessible and well-written introductions to algebraic set theory are \cite{awodey08, awodeyetal07, simpson05}.} In the first paper in this series \cite{bergmoerdijk08} we discussed how these predicative categories with small maps provide a sound and complete semantics for constructive set theory. In the second one \cite{bergmoerdijk07c}, we explained how realizability extensions of such predicative categories with small maps can be constructed. The purpose of the present paper is to do the same for sheaf-theoretic extensions. This program was summarised in \cite{bergmoerdijk07a}, where we announced the results that we will present and prove here.

For the convenience of the reader, and also to allow a comparison with the work by other researchers, we outline the main features of our approach. As said, the central concept in our theory is that of a predicative category with small maps. It axiomatises the idea of a category whose objects are classes and whose morphisms are functions between classes, and which is moreover equipped with a designated class of maps. The maps in the designated class are called small, and the intuitive idea is that the fibres of these maps are sets (in a certain axiomatic set theory). Such categories are in many ways like toposes, and to a large extent the purpose of our series of papers is to develop a topos theory for these categories. Indeed, like toposes, predicative categories with small maps turn out to be closed under realizability and sheaves. 

On the other hand, where toposes can be seen as models of a typed version of (constructive) higher-order arithmetic, predicative categories with small maps provide models of (constructive) \emph{set theories}. Furthermore, the notion of a predicative category with small maps is proof-theoretically rather weak: this allows us to model set theories which are proof-theoretically weaker than higher-order arithmetic, such as Aczel's set theory {\bf CZF} (see \cite{aczel78}). But at the same time, the notion of a predicative category with small maps can also be strengthened, so
that it leads to models of set theories proof-theoretically stronger than higher-order arithmetic, like {\bf IZF}. The reason for this is that one can impose additional axioms on the  class of small maps. This added flexibility is an important feature of algebraic set theory.

A central result in algebraic set theory says that the semantics
provided by predicative categories with small maps is complete. More precisely, every predicative category with small maps contains an object (``the initial ZF-algebra'' in the terminology of \cite{joyalmoerdijk95}, or ``the initial \spower-algebra'' in the terminology of \cite{bergmoerdijk08}\footnote{Appendix A in \cite{joyalmoerdijk95} contains a proof of the fact that both these terms refer to the same object. In the sequel we will use these terms interchangably.}) which carries the structure of a model of set theory. Which set-theoretic axioms hold in this model depends on the properties of the class of small maps and on the logic of the underlying category: in different situations, this initial ZF-algebra can be a model  of {\bf CZF}, of {\bf IZF}, or of ordinary {\bf ZF}. (The axioms of the constructive set theores {\bf CZF} and {\bf IZF} are recalled in Section 2 below.) The completeness referred to above results from the fact that from the syntax of {\bf CZF} (or {\bf (I)ZF} ), we can build a predicative category with small maps with the property that in the initial ZF-algebra in this category, precisely those sentences are valid which are derivable from the axiom of {\bf CZF} (see \cite{bergmoerdijk08}). (Completeness theorems of this kind go back to \cite{simpson99, awodeyetal07}. One should also mention that one can obtain a predicative category with small maps from the syntax of Martin-L\"of type theory: Aczel's interpretation of {\bf CZF} in Martin-L\"of type theory goes precisely via the initial ZF-algebra in this category. In fact, our proof of the existence of the initial ZF-algebra in any predicative category with small maps in \cite{bergmoerdijk08} was modelled on Aczel's interpretation, as it was in \cite{moerdijkpalmgren02}.)

In algebraic set theory we approach the construction of  realizability categories and of categories of sheaves in a
topos-theoretic spirit; that is, we regard these realizability and sheaf constructions as \emph{closure properties} of predicative categories with small maps. For realizability this means that starting from any predicative category with small maps $(\ct{E}, \smallmap{S})$ one can build a predicative realizability category with small maps $(\Eff_\ct{E}, \Eff_\smallmap{S})$ over it. Inside both of these categories, we have models of constructive set theory ({\bf CZF} say), as shown in the following picture. Here, the vertical arrows are two instances of the same construction of the initial ZF-algebra, applied to different predicative categories with small maps:
\diag{ (\ct{E}, \smallmap{S}) \ar[d] \ar[r] & (\Eff_\ct{E},
\Eff_\smallmap{S}) \ar[d] \\
\mbox{model of {\bf CZF}} \ar[r] & \mbox{realizability model of {\bf CZF}} }

Traditional treatments of realizability either regard it as a
model-theoretic construction (which would correspond to the lower edge of the diagram), or as a proof-theoretic interpretation (defining a realizability model of {\bf CZF} inside {\bf CZF}, as in \cite{rathjen06}, for instance): the latter would correspond to the left-hand vertical arrow in the special case where  $\ct{E}$ is the syntactic category associated to {\bf CZF}. So in a way our treatment captures both constructions in a uniform way. 

That realizability is indeed a closure property of predicative categories with small maps was the principal result of \cite{bergmoerdijk08}. The main result of the present paper is that the same is true for sheaves, leading to an analogous diagram:
\diag{ (\ct{E}, \smallmap{S}) \ar[d] \ar[r] & ({\rm Sh}_\ct{E},
{\rm Sh}_\smallmap{S}) \ar[d] \\
\mbox{model of {\bf CZF}} \ar[r] & \mbox{sheaf model of {\bf CZF}} }
The main technical difficulty in showing that predicative categories with small maps are closed under sheaves lies in showing that the axioms concerning inductive types (W-types) and an axiom called ``fullness'' (needed to model the subset collection axiom of {\bf CZF}) are inherited by sheaf models. The proofs of these facts are quite long and involved, and take up a large part of this paper (the situation for realizability was very similar). 

To summarise, in our approach there is one uniform construction of a model out of a predicative category with small maps $(\ct{E}, \smallmap{S})$, which one can apply to different kinds of such categories, constructed using syntax, using realizability, using sheaves, or any iteration or combination of these techniques.

We proceed to compare our results with those of other authors. Early work on categorical semantics of set theory (for example, \cite{freyd80} and \cite{fourman80}) was concerned with sheaf and realizability toposes defined over $\Sets$. The same applies to the book which introduced algebraic set theory \cite{joyalmoerdijk95}. In particular, to the best of our knowledge, before our work a systematic account was lacking of iterations and combinations of realizability and sheaf interpretations. In addition these earlier papers were concerned exclusively with impredicative set theories, such as {\bf ZF} or {\bf IZF}: the only exception seems to have been an early paper \cite{grayson83} by Grayson, treating models of predicative set theory in the context of what would now be called formal topology.

The first paper extending the methods of algebraic set theory to
predicative systems was \cite{moerdijkpalmgren02}. The authors of this paper showed how categorical models of Martin-L\"of type theory (with universes) lead to models of {\bf CZF} extended with a choice principle, which they dubbed the Axiom of Multiple Choice {\bf (AMC)}. They established how such categorical models of type theory are closed under sheaves, hence leading to sheaf models of a strengthening of {\bf CZF}. They did not develop a semantics for {\bf CZF} \emph{per se} and relied on a technical notion of a collection site, which we manage to avoid here (moreover, there was a mistake in their treatment 
of W-types of sheaves; we correct this in Section 4.4 below, see also \cite{bergmoerdijk08b}). 

Two accounts of presheaf models in the context of algebraic set theory have been written by Gambino \cite{gambino05} and Warren \cite{warren07}. In \cite{gambino05} Gambino shows how an earlier (unpublished) construction of a model of constructive set theory by Dana Scott can be regarded as an initial ZF-algebra in a category of presheaves, and that one can perform the construction in a predicative metatheory as well. Warren shows in \cite{warren07} that many of the axioms that we will discuss  are inherited by categories of coalgebras for a Cartesian comonad, a construction which includes presheaf models as
a special case. But note that neither of these authors discusses the technically complicated axioms concerning W-types and fullness, as we will do in Sections 3 and 4 below. 

In his PhD thesis \cite{gambino02}, Gambino gave a systematic
account of Heyting-valued models for {\bf CZF} (see also
\cite{gambino06}). This work was in the context of formal topology (essentially, sites whose underlying categories are posets). He has subsequently worked on generalising this to arbitrary sites and on putting this in the context of algebraic set theory. In \cite{gambino08}, he took the first step in constructing the sheafification functor and in  \cite{awodeygambinolumsdainewarren09}, written together with Awodey, Lumsdaine and Warren, he checks that the basic axioms for small maps are inherited by categories of sheaves in the general setting of sheaves for a Lawvere-Tierney topology. We will extend these results by proving that for sites which have a presentation (for a definition, see \refdefi{site} below), the axioms for W-types and for fullness are stable under taking sheaf extensions. Note that for proof-theoretic reasons, fullness cannot be stable under taking more general kinds of sheaves such as those for a site which does not have a presentation, or for a Lawvere-Tierney topology. The point is that {\bf CZF} extended with the Law of Excluded Middle gives {\bf ZF}, a much stronger system proof-theoretically, and  therefore a double-negation interpretation of {\bf CZF} in itself must fail. The culprit turns out to be the fullness axiom, which can therefore not be stable under taking sheaves for the double-negation topology or sheaves for an arbitrary site (see \cite{gambino06} and \cite{grayson83}). 

We conclude this introduction by outlining the organisation of our paper. In Section 2 we recall the main definitions from \cite{bergmoerdijk07a,bergmoerdijk08}. We will introduce the axioms for a class of small maps necessary to obtain models of {\bf CZF} and {\bf IZF}. Among these necessary axioms, we will discuss the fullness axiom, the axioms concerning W-types and the axiom of multiple choice in detail, as these are the most complicated technically and our main results, which we formulate precisely in Section 2.5, are concerned with these axioms.

In Section 3 we show that predicative categories with small maps are closed under presheaves and that all the axioms that we have listed in Section 2 are inherited by such presheaf models. An important part of our treatment is that we distinguish between two classes of small maps: the ``pointwise'' and ``locally'' small ones. It turns out that for certain axioms it is easier to show that they are inherited by pointwise small maps while for other axioms it is easier to show that they are inherited by locally small maps, and therefore it is an important result that these classes of maps coincide.

We follow a similar strategy in Section 4, where we discuss sheaves: we again distinguish between two classes of maps, where for some axioms it is easier to use one definition, while for other axioms it turns out to be easier to use the other. To show that these two classes coincide we use the fullness axiom and assume that the site has a presentation.\footnote{In \cite{bergmoerdijk07a} we claimed that (instead of fullness) the
exponentiation axiom would suffice to establish this result, but that might not be correct.} This section also contains our main technical results: that sheaf models inherit the fullness axiom, as well as the axioms concerning W-types.\footnote{One subtlety arises when we try to show that an axiom saying that certain inductives types are small (axiom {\bf (WS)} to be precise) is inherited by sheaf models: we show this using the axiom of multiple choice. In fact, we suspect that something of this sort is unavoidable and one has to go beyond {\bf CZF} proper to show that its validity is inherited by sheaf models.} Strictly speaking our results for presheaves in Section 3 are special
cases of our results in Section 4. We believe, however, that it is useful to give direct proofs of the results for presheaves,  and in many cases it is helpful to see how the proof goes in the (easier) presheaf case before embarking on the more involved proofs in the sheaf case.

Finally, in Section 5 we give explicit descriptions of the sheaf models of constructive set theory our results lead to. We also point out the connection to forcing for classical set theories. 

This will complete our program for developing an abstract semantics of constructive set theory, in particular of Aczel's {\bf CZF}, as outlined \cite{bergmoerdijk07a}. As a result topos-theoretic insights and categorical methods can now be used in the study of constructive set theories. For instance, one can obtain consistency and independence results using sheaf and realizability models or by a combination of these interpretations. In future work, we will use sheaf-theoretic methods to show that the fan rule as well as certain continuity rules are derived rules for {\bf CZF} and related theories \cite{bergmoerdijk11}.

The main results of this paper were presented by the second author in a tutorial on categorical logic at the Logic Colloquium 2006 in Nijmegen. We are grateful to the organisers of the Logic Colloquium for giving one of the authors this opportunity. The final draft of this paper was completed during a stay of the first author at the Mittag-Leffler Institute in Stockholm. We would like to thank the Institute and the organisers of the program in Mathematical Logic in Fall 2009 for awarding him a grant which enabled him to complete this paper in such excellent working conditions. In addition, we would like to acknowledge the helpful discussions we had with Steve Awodey, Nicola Gambino, Jaap van Oosten, Erik Palmgren, Thomas Streicher, Michael Warren, and especially Peter LeFanu Lumsdaine (see \refrema{errorinMP} below). 

\section{Preliminaries}

\subsection{Review of Algebraic Set Theory}

In this section we recall the main features of our approach to Algebraic Set Theory from \cite{bergmoerdijk07a,bergmoerdijk08}.

We will always assume that our ambient category \ct{E} is a \emph{positive Heyting category}. That means that \ct{E} is
\begin{enumerate}
\item[(i)] Cartesian, i.e., it has finite limits.
\item[(ii)] regular, i.e., morphisms factor in a stable fashion as a cover followed by a monomorphism.\footnote{Recall that a map $f: B \to A$ is a cover, if the only subobject of $A$ through which it factors, is the maximal one; and that $f$ is a regular epimorphism if it is the coequalizer of its kernel pair. These two classes coincide in regular categories (see \cite[Proposition A1.3.4]{johnstone02a}).}
\item[(iii)] positive, i.e., it has finite sums, which are disjoint and stable.
\item[(iv)] Heyting, i.e., for any morphism \func{f}{Y}{X} the induced pullback functor \func{f^*}{{\rm Sub}(X)}{{\rm Sub}(Y)} has a right adjoint $\forall_f$.
\end{enumerate}
This means that \ct{E} is rich enough to interpret first-order intuitionistic logic. Such a category \ct{E} will be called a \emph{category with small maps}, if it comes equipped with a class of maps \smallmap{S} satisfying a list of axioms. To formulate these, we use the notion of a covering square. 

\begin{defi}{coveringsquare}
A diagram in \ct{E} of the form
\diag{ D \ar[d]_f \ar[r] &  C \ar[d]^g \\
B \ar[r]_p & A}
is called a \emph{quasi-pullback}, when the canonical map $D \rTo B \times_A C$ is a cover. If $p$ is also a cover, the diagram will be called a \emph{covering square}. When $f$ and $g$ fit into a covering square as shown, we say that $f$ \emph{covers} $g$, or that $g$ \emph{is covered by} $f$.
\end{defi}

\begin{defi}{smallmaps}
A class of maps in \ct{E} satisfying the following axioms {\bf (A1-9)} will be called a \emph{class of small maps}:
\begin{description}
\item[(A1)] (Pullback stability) In any pullback square
\diag{ D \ar[d]_g \ar[r] & B \ar[d]^f \\
C \ar[r]_p & A }
where $f \in \smallmap{S}$, also $g \in \smallmap{S}$.
\item[(A2)] (Descent) If in a pullback square as above $p$ is a cover and $g \in \smallmap{S}$, then also $f \in \smallmap{S}$.
\item[(A3)] (Sums) Whenever $X \rTo Y$ and $X'\rTo Y'$ belong to \smallmap{S}, so does $X + X' \rTo Y + Y'$.
\item[(A4)] (Finiteness) The maps $0 \rTo 1, 1 \rTo 1$ and $1+1 \rTo 1$ belong to \smallmap{S}.
\item[(A5)] (Composition) $\smallmap{S}$ is closed under composition.
\item[(A6)] (Quotients) In a commuting triangle
\diag{ Z \ar[dr]_h \ar@{->>}[rr]^f & & Y \ar[dl]^g \\
& X, &  }
if $f$ is a cover and $h$ belongs to \smallmap{S}, then so does $g$.
\item[(A7)] (Collection) Any two arrows \func{p}{Y}{X} and \func{f}{X}{A} where $p$ is a cover and $f$ belongs to \smallmap{S} fit into a covering square
\diag{ Z \ar[d]_g \ar[r] & Y \ar@{->>}[r]^p & X \ar[d]^f \\
B \ar@{->>}[rr]_h & & A,}
where $g$ belongs to \smallmap{S}.
\item[(A8)] (Heyting) For any morphism \func{f}{Y}{X} belonging to \smallmap{S}, the right adjoint to pullback
\[ \func{\forall_f}{{\rm Sub}(Y)}{{\rm Sub}(X)} \]
sends small monos to small monos.
\item[(A9)] (Diagonals) All diagonals \func{\Delta_X}{X}{X \times X} belong to \smallmap{S}.
\end{description}
For further discussion of these axioms we refer to \cite{bergmoerdijk08}.

A pair $(\ct{E}, \smallmap{S})$ in which \smallmap{S} is a class of small maps in \ct{E} will be called a \emph{category with small maps}. In such categories with small maps, objects $A$ will be called \emph{small}, if the unique map from $A$ to the terminal object is small. A subobject $A \subseteq X$ will be called a \emph{small subobject} if $A$ is a small object. If any of its representing monomorphisms $m: A \to X$ is small, they all are and in this case the subobject will be called \emph{bounded}.
\end{defi}

\begin{rema}{stabunderslicing}
In the sequel we will often implicitly use that categories with small maps are stable under slicing. By this we mean that for any category with small maps $(\ct{E}, \smallmap{S})$ and object $X$ in \ct{E}, the pair $(\ct{E}/X, \smallmap{S}/X)$, with $\smallmap{S}/X$ being defined by
\[ f \in \smallmap{S}/X \Leftrightarrow \Sigma_X f \in \smallmap{S}, \] 
is again a category with small maps (here $\Sigma_X$ is the forgetful functor $\ct{E}/X \to \ct{E}$ sending an object $p: A \to X$ in $\ct{E}/X$ to $A$ and morphisms to themselves). Moreover, any of the further axioms for classes of small maps to be introduced below are stable under slicing, in the sense that their validity in the slice over 1 implies their validity in every slice.
\end{rema}

\begin{rema}{bounded separation}
A very useful feature of categories of small maps, and one we will frequently exploit, is that they satisfy an internal form of bounded separation. A precise statement is the following: if $\phi(x)$ is a formula in the internal logic of \ct{E} with free variable $x \in X$, all whose basic predicates are interpreted as bounded subobjects (note that this includes all equalities, by {\bf (A9)}), and which contains existential and universal quantifications $\exists_f$ and $\forall_f$ along small maps $f$ only, then
\[ A = \{ x \in X \, : \, \phi(x) \} \subseteq X \]
defines a bounded subobject of $X$. In particular, smallness of $X$ implies smallness of $A$. 
\end{rema}

\begin{defi}{predcatwsmallmaps}
A category with small maps $(\ct{E}, \smallmap{S})$ will be called a \emph{predicative category with small maps}, if the following axioms hold: 
\begin{description}
\item[($\Pi$E)] All morphisms $f \in \smallmap{S}$ are exponentiable.
\item[(WE)] For all $\func{f}{X}{Y} \in \smallmap{S}$, the W-type $W_f$ associated to $f$ exists.
\item[(NE)] \ct{E} has a natural numbers object $\NN$.
\item[(NS)] Moreover, $\NN \rTo 1 \in \smallmap{S}$.
\item[(Representability)] There is a small map \func{\pi}{E}{U} (the ``universal small map'') such that any $\func{f}{Y}{X} \in \smallmap{S}$ fits into a diagram of the form
\diag{Y \ar[d]_f & B \ar[d] \ar[r] \ar@{->>}[l] & E \ar[d]^{\pi} \\
X & A \ar[r] \ar@{->>}[l] & U, }
where the left hand square is covering and the right hand square is a pullback.
\item[(Bounded exactness)] For any equivalence relation \diag{R \ar@{ >->}[r] & X \times X} given by a small mono, a stable quotient $X/R$ exists in \ct{E}.
\end{description}
(For a detailed discussion of these axioms we refer again to \cite{bergmoerdijk08};  W-types and the axiom {\bf (WE)} will also be discussed in Section 2.3 below.)
\end{defi}

In predicative categories with small maps one can derive the existence of a power class functor, classifying small subobjects:
\begin{defi}{powerobj}
By a \emph{$D$-indexed family of subobjects} of $C$, we mean a subobject $R \subseteq C \times D$. It will be called a \emph{$D$-indexed family of small subobjects}, whenever the composite
\[ R \subseteq C \times D \rTo D \]
belongs to \smallmap{S}. If it exists, the \emph{power class object} $\spower X$ is the classifying object for the families of small subobjects of $X$. This means that it comes equipped with a $\spower X$-indexed family of small subobjects of $X$, denoted by $\in_X \subseteq X \times \spower X$ (or simply $\in$, whenever $X$ is understood), with the property that for any $Y$-indexed family of small subobjects of $X$, $R \subseteq X \times Y$ say, there exists a unique map \func{\rho}{Y}{\spower X} such that the square
\diag{ R \ar@{ >->}[d] \ar[r] & \in_X \ar@{ >->}[d] \\
X \times Y \ar[r]_{\id \times \rho} & X \times \spower X}
is a pullback.
\end{defi}
\begin{prop}{existenceofpowerobj} {\rm \cite[Corollary 6.11]{bergmoerdijk08}}
In a predicative category with small maps all power class objects exist.
\end{prop}
Moreover, one can show that the assignment $X \mapsto \spower X$ is functorial and that this functor has an initial algebra.
\begin{theo}{existenceofinitZFalgebra} {\rm \cite[Theorem 7.4]{bergmoerdijk08}}
In a predicative category with small maps the $\spower$-functor has an initial algebra.
\end{theo}
The importance of this result resides in the fact that this initial algebra can be used to model a weak intuitionistic set theory: if $V$ is the initial algebra and $E: V \to \spower V$ is the inverse of the $\spower$-algebra map on $V$ (which is an isomorphism, since $V$ is an initial algebra), then one can define a binary predicate $\epsilon$ on $V$ by setting
\[ x \epsilon y \Leftrightarrow x \in_V E(y), \]
where $\in_V \subseteq V \times \spower V$ derives from the power class structure on $\spower V$. The resulting structure $(V, \epsilon)$ models a weak intuitionistic set theory, which we have called {\bf RST} (for rudimentary set theory), consisting of the following axioms:
\begin{description}
\item[Extensionality:] $\forall x \, ( \, x \epsilon a \leftrightarrow x \epsilon b \, ) \rightarrow a = b$.
\item[Empty set:] $\exists x  \, \forall y  \, \lnot y \epsilon x $.
\item[Pairing:] $\exists x \, \forall y \, (\,  y \epsilon x \leftrightarrow y = a \lor y = b \, )$.
\item[Union:] $\exists x  \, \forall y \, ( \, y \epsilon x \leftrightarrow \exists z \epsilon a \, y \epsilon z  \, )$.
\item[Set induction:] $\forall x \, (\forall y  \epsilon x \, \phi(y) \rightarrow \phi(x)) \rightarrow \forall x \, \phi(x)$.
\item[Bounded separation:] $\exists x \,  \forall y \, ( \, y \epsilon x \leftrightarrow y \epsilon a \land \phi(y) \, ) $, for any bounded formula $\phi$ in which $a$ does not occur.
\item[Strong collection:] $\forall x \epsilon a \, \exists y \, \phi(x,y) \rightarrow \exists b \, \mbox{B}(x \epsilon a, y \epsilon b) \, \phi$, where $\mbox{B}(x \epsilon a, y \epsilon b) \, \phi$ abbreviates
\[ \forall x \epsilon a \, \exists y \epsilon b \, \phi \land \forall y \epsilon b \, \exists x \epsilon a \, \phi. \]
\item[Infinity:] $\exists a \, ( \, \exists  x \, x \epsilon a \, ) \land ( \, \forall x  \epsilon a \, \exists y \epsilon a \, x \epsilon y \, )$.
\end{description}
In fact, as shown in \cite{bergmoerdijk08}, the initial $\spower$-algebras in predicative categories with small maps form a complete semantics for the set theory {\bf RST}. To obtain complete semantics for better known intuitionistic set theories, like {\bf IZF} and {\bf CZF}, one needs further requirements on the class of small maps \smallmap{S}. For example, the set theory {\bf IZF} is obtained from {\bf RST} by adding the axioms
\begin{description}
\item[Full separation:] $\exists x \,  \forall y \, ( \, y \epsilon x \leftrightarrow y \epsilon a \land \phi(y) \, ) $, for any formula $\phi$ in which $a$ does not occur.
\item[Power set:] $\exists x \, \forall y \, ( \, y \epsilon x \leftrightarrow y \subseteq a \, )$, where $y \subseteq a$ abbreviates $\forall z \, ( z \epsilon y \rightarrow z \epsilon a)$.
\end{description}
And to obtain a sound and complete semantics for {\bf IZF} one requires of ones predicative category of small maps that it satisfies:
\begin{description}
\item[(M)] All monomorphisms belong to \smallmap{S}.
\item[(PS)] For any map $\func{f}{Y}{X} \in \smallmap{S}$, the power class object $\slspower{X} (f) \rTo X$ in $\ct{E}/X$ belongs to \smallmap{S}.
\end{description}
The set theory {\bf CZF}, introduced by Aczel in \cite{aczel78}, is obtained by adding to {\bf RST} a weakening of the power set axiom called subset collection:
\begin{description}
\item[Subset collection:] $\exists c \, \forall z \, ( \forall x \epsilon a \, \exists y \epsilon b \, \phi(x,y,z) \rightarrow \exists d \epsilon c \, \mbox{B}(x \epsilon a, y \epsilon d) \, \phi(x, y, z)) $.
\end{description}
For a suitable categorical analogue, see Section 2.3 below.

For the sake of completeness we also list the following two axioms, saying that certain $\Pi$-types and W-types are small. (The first therefore corresponds to the exponentiation axiom in set theory; we will say more about the second in Section 2.2 below.)
\begin{description}
\item[($\Pi$S)] For any map $\func{f}{Y}{X} \in \smallmap{S}$, a functor 
\[ \func{\Pi_f}{\ct{E}/Y}{\ct{E}/X} \]
right adjoint to pullback exists and preserves morphisms in \smallmap{S}.
\item[(WS)] For all $\func{f}{X}{Y} \in \smallmap{S}$ with $Y$ small, the W-type $W_f$ associated to $f$ is small.
\end{description}

\subsection{W-types}

In a predicative category with small maps $(\ct{E}, \smallmap{S})$ the axiom {\bf ($\Pi$E)} holds and therefore any small map \func{f}{B}{A} is exponentiable. It therefore induces an endofunctor on \ct{E}, which will be called the \emph{polynomial functor} $P_f$ associated to $f$. The quickest way to define it is as the following composition:
\diag{ \ct{C} \cong \ct{C}/1 \ar[r]^(.6){B^*} & \ct{C}/B \ar[r]^{\Pi_f} & \ct{C}/A \ar[r]^(.4){\Sigma_A} & \ct{C}/1 \cong \ct{C}. }
In more set-theoretic terms it could be defined as:
\[ P_f(X) = \sum_{a \in A} X^{B_a}. \]
Whenever it exists, the initial algebra for the polynomial functor $P_f$ will be called the \emph{W-type associated to} $f$.

Intuitively, elements of a W-type are well-founded trees. In the category of sets, all W-types exist, and the W-types have as elements well-founded trees, with an appropriate labelling of its edges and nodes. What is an appropriate labelling is determined by the branching type \func{f}{B}{A}: nodes should be labelled by elements $a \in A$, edges by elements $b \in B$, in such a way that the edges into a node labelled by $a$ are uniquely enumerated by $f^{-1}(a)$. The following picture hopefully conveys the idea:
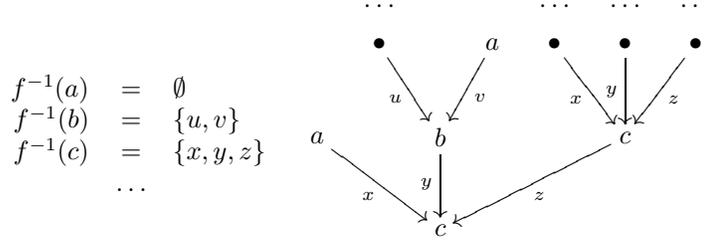
\begin{displaymath}
\xymatrix@C=.75pc@R=.5pc{ & & \ldots & & & \ldots & \ldots & \ldots \\
           & & {\bullet} \ar[dr]_u & & a \ar[dl]^v & {\bullet} \ar[dr]_x & {\bullet} \ar[d]_y & {\bullet} \ar[dl]^z \\
*{\begin{array}{rcl}
f^{-1}(a) & = & \emptyset \\
f^{-1}(b) & = & \{ u, v \} \\
f^{-1}(c) & = & \{ x, y, z \} \\
& \ldots & 
\end{array}}       & a \ar[drr]_x & & b \ar[d]_y & & & c \ar[dlll]^z  \\
           & & & c & & & & } 
\end{displaymath}
This set has the structure of a $P_f$-algebra: when an element $a \in A$ is given, together with a map \func{t}{B_a}{W_f}, one can build a new element ${\rm sup}_a t \in W_f$, as follows. First take a fresh node, label it by $a$ and draw edges into this node, one for every $b \in B_a$, labelling them accordingly. Then on the edge labelled by $b \in B_a$, stick the tree $tb$. Clearly, this sup operation is a bijective map. Moreover, since every tree in the W-type is well-founded, it can be thought of as having been generated by a possibly transfinite number of iterations of this sup operation. That is precisely what makes this algebra initial. The trees that can be thought of as having been used in the generation of a certain element $w \in W_f$ are called its subtrees. One could call the trees $tb \in W_f$ the \emph{immediate subtrees} of ${\rm sup}_a t$, and $w' \in W_f$ a \emph{subtree} of $w \in W_f$ if it is an immediate subtree, or an immediate subtree of an immediate subtree, or\ldots, etc. Note that with this use of the word subtree, a tree is never a subtree of itself (so proper subtree might have been a better terminology).

We recall that there are two axioms concerning W-types:
\begin{description}
\item[(WE)] For all $\func{f}{X}{Y} \in \smallmap{S}$, the W-type $W_f$ associated to $f$ exists.
\item[(WS)] Moreover, if $Y$ is small, also $W_f$ is small.
\end{description}
Maybe it is not too late to point out the following fact, which explains why these axioms play no essential role in the impredicative setting:
\begin{theo}{WEandWSinimpred}
Let $(\ct{E}, \smallmap{S})$ be a category with small maps satisfying {\bf (NS)} and {\bf (M)}.
\begin{enumerate}
\item If \smallmap{S} satisfies {\bf (PE)}, then it also satisfies {\bf (WE)}.
\item If \smallmap{S} satisfies {\bf (PS)}, then it also satisfies {\bf (WS)}.
\end{enumerate}
\end{theo}
\begin{proof}
Note that in a category with small maps satisfying {\bf (M)} and {\bf (PE)} the object $\spower(1)$ is a subobject classifier. Therefore the first result can be shown along the lines of Chapter 3 in \cite{joyalmoerdijk95}. For showing the second result, one simply copies the argument why toposes with nno have all W-types from \cite{moerdijkpalmgren02}.
\end{proof}

In the sequel we will need the following result. We will write $\inhspower X$ for the object of small inhabited subobjects of $X$:
\[ \inhspower X = \{ A \in \spower X \, : \, \exists x \in X \, ( x \in A) \}. \]
\begin{theo}{initialalginpredcatofsmallmaps}
For any small map $f: B \to A$ in a predicative category with small maps $(\ct{E}, \smallmap{S})$, the endofunctors on \ct{E} defined by
\[ \Phi = P_f \circ \spower \qquad \mbox{and} \qquad \Psi = P_f \circ \inhspower \]
have initial algebras.
\end{theo}
\begin{rema}{oninitalgforphiandpsi}
Before we sketch the proof of \reftheo{initialalginpredcatofsmallmaps}, it might be good to explain the intuitive meaning of these initial algebras. In fact, they are variations on the W-types explained above: they are also classes of well-founded trees, but the conditions on the labellings of the nodes and edges are slightly different. It is still the case that nodes are labelled by elements $a \in A$ and edges with elements $b \in B$, in such a way that if $b \in B$ decorates a certain edge, then $f(b)$ decorates the node it points to. But whereas in a W-type, every node in a well-founded tree labelled with $a \in A$ has for every $b \in f^{-1}(a)$ \emph{precisely one} edge into it labelled with $b$, in the initial algebras for $\Phi$ there are \emph{set-many, and possibly none}, and in the initial algebra for $\Psi$ there are \emph{set-many, but at least one}.
\end{rema}
\begin{proof}
The proof of \reftheo{initialalginpredcatofsmallmaps} is a variation on that of Theorem 7.4 in \cite{bergmoerdijk08} and therefore we will only sketch the argument.

Fix a universal small map $\pi: E \to U$, and write 
\[ S = \{ (a \in A, u \in U, \phi: E_u \to B_a)  \}. \]
Let ${\cal K}$ be the W-type in \ct{E} associated to the map $g$ fitting into the pullback square
\diag{ R \ar[d]_g \ar[r] & E \ar[d]^{\pi} \\
S \ar[r]_{\rm{proj}} & U. }
An element $k \in {\cal K}$ is therefore of the form ${\rm sup}_{(a, u, \phi)}t$, where $(a,  u, \phi) \in S$ is the label of the root of $k$ and $t: E_u \to {\cal K}$ is the function that assigns to every element $e \in E_u$ the tree that is attached to the root of $k$ with the edge labelled with $e$. Define the following equivalence relation on ${\cal K}$ by recursion: ${\rm sup}_{(a, u, \phi)}t \sim {\rm sup}_{(a', u', \phi')}t'$, if $a = a'$ and 
\begin{quote}
for all $e \in E_u$ there is an $e' \in E_{u'}$ such that $\phi(e) = \phi'(e') \mbox{ and } t(e) \sim t'(e')$, and for all $e' \in E_{u'}$ there is an $e \in E_{u}$ such that $\phi(e) = \phi'(e') \mbox{ and } t(e) \sim t'(e')$.
\end{quote}
(The existence of this relation $\sim$ can be justified using the methods of \cite{berg05} or \cite{bergmoerdijk08}. See Theorem 7.4 in \cite{bergmoerdijk08}, for instance.) The equivalence relation is bounded (one proves this by induction) and its quotient is the initial algebra for $\Phi$.

The initial algebra for $\Psi$ is constructed in the same way, but with $S$ defined as
\[ S = \{ (a \in A, u \in U, \phi: E_u \to B_a)  \, : \, \phi \mbox{ is a cover} \}. \]
\end{proof}

\subsection{Fullness}

In order to express the subset collection axiom, introduced by Peter Aczel in \cite{aczel78}, in diagrammatic terms, it is helpful to consider an axiom which is equivalent to it called \emph{fullness} (see \cite{aczelrathjen01}). In the language of set theory one can formulate fullness using the notion of a \emph{multi-valued section}: a multi-valued section (or \emph{mvs}) of a function \func{\phi}{b}{a} is a multi-valued function $s$ from $a$ to $b$ such that $\phi s = \id_a$ (as relations). Identifying $s$ with its image, this is the same as a subset $p$ of $b$ such that $p \subseteq b \rTo a$ is surjective. For us, fullness states that for any such $\phi$ there is a small family of \emph{mvs}s such that any \emph{mvs} contains one in this family. Written out formally:
\begin{description}
\item[Fullness:] $\exists z\;\! (z \;\!\subseteq \;\!{\bf mvs}(\phi) \land \forall x \;\!\epsilon \;\!{\bf mvs}(\phi) \;\! \exists c \;\!\epsilon \;\!z \;\! (c \;\!\subseteq \;\!x))$.
\end{description}
Here, ${\bf mvs}(\phi)$ is an abbreviation for the class of all multi-valued sections of a function \func{\phi}{b}{a}, i.e., subsets $p$ of $b$ such that $\forall x \epsilon a \, \exists y \epsilon p \, \phi(y) = x$. 

In order to reformulate this diagrammatically, we say that a multi-valued section (\emph{mvs}) for a small map \func{\phi}{B}{A}, over some object $X$, is a subobject $P \subseteq B$ such that the composite $P \rTo A$ is a \emph{small} cover. (Smallness of this map is equivalent to $P$ being a bounded subobject of $B$.) We write
\[ {\rm mvs}_X(\phi) \]
for the set of all \emph{mvs}s of a map $\phi$. This set obviously inherits the structure of a partial order from Sub($B$). Note that any morphism \func{f}{Y}{X} induces an order-preserving map
\[ f^*: {\rm mvs}_X(\phi) \rTo {\rm mvs}_Y (f^*\phi), \]
obtained by pulling back along $f$. To avoid overburdening the notation, we will frequently talk about the map $\phi$ over $Y$, when we actually mean the map $f^* \phi$ over $Y$, the map $f$ always being understood. 

The categorical fullness axiom now reads:
\begin{description}
\item[(F)] For any $\func{\phi}{B}{A} \in \smallmap{S}$ over some $X$ with $A \rTo X \in \smallmap{S}$, there is a cover \func{q}{X'}{X} and a map \func{y}{Y}{X'} belonging to $\smallmap{S}$, together with an \emph{mvs} $P$ of $\phi$ over $Y$, with the following ``generic'' property: if \func{z}{Z}{X'} is any map and $Q$ any \emph{mvs} of $\phi$ over $Z$, then there is a map \func{k}{U}{Y} and a cover \func{l}{U}{Z} with $yk = zl$ such that $k^* P \leq l^* Q$ as \emph{mvs}s of $\phi$ over $U$.
\end{description}

It is easy to see that in a set-theoretic context fullness is a consequence of the powerset axiom (because then the collection of \emph{all} multi-valued sections of a map $\phi: b \to a$ forms a set) and implies the exponentiation axiom (because if $z$ is a set of \emph{mvs}s of the projection $p: a \times b \to a$ such that any \emph{mvs} is refined by one is this set, then the set of functions from $a$ to $b$ can be constructed from $z$ by selecting the univalued elements, i.e., those elements that are really functions). Showing that in a categorical context {\bf (F)} follows from {\bf (PS)} and implies {\bf ($\Pi$S)} is not much harder and we will therefore not write out a formal proof. 

In the sequel we will use the following two lemmas concerning the fullness axiom:
\begin{lemm}{diffgenmvss}
Suppose we have the following diagram
\diag{ Y_2 \ar[d]_{f_2} \ar[r]^{\beta} & Y_1 \ar[d]^{f_1} \\
X_2 \ar[r] \ar@{->>}[rd] & X_1 \ar@{->>}[d] \\
& X,}
in which the square is a quasi-pullback and $f_1$ and $f_2$ are small. When $P$ is a ``generic'' \emph{mvs} for a map \func{\phi}{B}{A} over $X$ living over $Y_1$ (``generic'' as in the statement of the fullness axiom), then $\beta^* P$ is also a generic \emph{mvs} for $\phi$, living over $Y_2$.
\end{lemm}
\begin{proof}
A simple diagram chase.
\end{proof}

\begin{lemm}{redlemmforfull}
Suppose we are given a diagram of the form
\diag{ B_0 \ar@{->>}[r] \ar[d]_{\psi} & B \ar[d]^{\phi} \\
A_0 \ar@{->>}[r] \ar[d]_i & A \ar[d]^j \\
X_0 \ar@{->>}[r]_p & X, }
in which both squares are covering and all the vertical arrows are small. If a generic  \emph{mvs} for $\psi$ exists over $X_0$, then also a generic \emph{mvs} for $\phi$ exists over $X$.
\end{lemm}
\begin{proof}
This was Lemma 6.23 in \cite{bergmoerdijk08}.
\end{proof}

\subsection{Axiom of multiple choice}

The axiom of multiple choice was introduced by Moerdijk and Palmgren in \cite{moerdijkpalmgren02}. Their motivation was to have a choice principle which is implied by the existence of enough projectives (``the presentation axiom'' in Aczel's terminology) and is stable under taking sheaves (unlike the existence of enough projectives). We will use it in Section 4.4 to show that the axiom {\bf(WS)} is stable under taking sheaves. 

One can give a succinct formulation of the axiom of multiple choice using the notion of a collection span (see \cite[Definition 6.14]{bergmoerdijk08}).\footnote{The way we formulate the Axiom of Multiple Choice here is slightly different from how it was stated in \cite{moerdijkpalmgren02}. Both formulations are equivalent, however; see \cite{bergmoerdijk11b}.}
\begin{defi}{collspan}
A span $(g, h)$ in \ct{E}
\diag{ C & D \ar[r]^h \ar[l]_g & B}
is called a \emph{collection span}, when, in the internal logic, it holds that for any map \func{f}{E}{D_c} covering some fibre of $g$, there is a fibre $D_{c'}$ of $g$ and a map  \func{p}{D_{c'}}{E} such that $fp$ is a cover over $B$. A collection span is $\ct{E}/A$ will be called a \emph{collection span over $A$}.

Diagrammatically, we can express this by asking that for any map \func{e}{E}{C} and any epi $F \rTo E \times _C D$ there is a diagram of the form
\diag{
& & B & & \\
D \ar[d]_g \ar@/^/[urr]^h & E' \times_C D \ar[r] \ar[l] \ar[d] & F \ar@{->>}[r] & E \times_C D \ar[r] \ar[d] & D \ar[d]^g \ar@/_/[ull]_h \\
C & E' \ar@{->>}[rr] \ar[l] & & E \ar[r]_c & C}
where the middle square is a covering square, involving the given map $F \rTo D \times _C E$, while the other two squares are pullbacks.
\end{defi}

\begin{description}
\item[(AMC)] (Axiom of multiple choice) For any small map \func{f}{Y}{X}, there is a cover \func{q}{A}{X} and a diagram 
\diag{ D \ar[d]_g \ar@{->>}[r]^h & q*Y \ar[d]_{q^*f} \ar@{->>}[r] & Y \ar[d]^f \\
C \ar@{->>}[r]_r & A \ar@{->>}[r]_q & X,}
in which the right square is a pullback and the left square a covering square in which all maps are small and in which $(g, h)$ is a collection span over $A$.
\end{description}
In the internal logic {\bf (AMC)} is often applied in the following form:
\begin{lemm}{AMCininternallogic}
In a predicative category with small maps in which {\bf (AMC)} holds, the following principle holds in the internal logic: any small map \func{f}{B}{A} between small objects fits into a covering square
\diag{ D \ar[r]^q \ar[d]_g & B \ar[d]^f \\
C \ar[r]_p & A }
in which all maps and objects are small and $(g, q)$ is a collection span over $A$.
\end{lemm}
\begin{proof}
This is proved exactly as Proposition 4.6 in \cite{moerdijkpalmgren02}.
\end{proof}

The following result was proved in \cite{moerdijkpalmgren02} as well. Recall from \cite{aczel86,aczelrathjen01} that the existence of many inductively defined sets within {\bf CZF} can be guaranteed, in a predicatively acceptable way, by extending {\bf CZF} with Aczel's Regular Extension Axiom. 
\begin{prop}{AMCgivesREA}
If $(\ct{E}, \smallmap{S})$ is a predicative category with small maps satisfying the axioms {\bf (AMC)}, {\bf ($\Pi$S)} and {\bf (WS)}, then Aczel's Regular Extension Axiom holds in the initial \spower-algebra in this category.
\end{prop}

In addition, we will need:

\begin{prop}{AMCgivesfullness}
Let $(\ct{E}, \smallmap{S})$ be a predicative category with small maps. If \smallmap{S} satisfies the axioms {\bf (AMC)} and {\bf ($\Pi$S)}, then it satisfies the axiom {\bf (F)} as well.
\end{prop}
\begin{proof}
We argue internally and use \reflemm{AMCininternallogic}. So suppose that {\bf (AMC)} holds and $f: B \to A$ is a small map between small objects. We need to find a small collection of \emph{mvs}s $\{ P_y \, : \, y \in Y \}$ such that any \emph{mvs} of $f$ is refined by one in this family.

We apply \reflemm{AMCininternallogic} to $A \to 1$ to obtain a covering square of the form
\diag{ D \ar[r]^h \ar[d] & A \ar[d] \\
C \ar[r] & 1, }
such that for any cover $p: E \to D_c$ we find a $c' \in C$ and a map $t: D_{c'} \to E$ such that $pt$ is a cover over $A$. Let $Y$ be the collection of all pairs $(c,s)$ with $c$ in $C$ and $s$ a map $D_c \to B$ such that $fs = h_c$, and let $P_y$ be the image of the map $s:D_c \to B$. Then $P_y$ is an \emph{mvs}, because the $h_c$ are epi, and $Y$ is small, because {\bf ($\Pi$S)} holds.

Now suppose $n: Q \to B$ is any mono such that $fn: Q \to B \to A$ is a cover. Pick a $c \in C$ and pull back $fn$ along $h_c$ to obtain a cover $q: E \to D_c$, as in:
\diag{ E \ar[r]^p \ar@{->>}[d]_q & Q \ar@{->>}[d]^{fn} \\
D_c \ar[r]_{h_c} & A. } 
It follows that there exists an element $c' \in C$ and a map $g: D_{c'} \to E$ such that $qg$ is a cover over $A$. Set $s = npg$ and $y = (c', s)$. Then $P_y = {\rm Im}(g)$ is contained in $Q$.
\end{proof}

\subsection{Main results}

After all these definitions, we can formulate our main result. Let ${\cal A}$ be either $\{ {\bf (F)} \}$, or $\{ {\bf (AMC), (\Pi S), (WS) } \}$, or $\{ {\bf (M), (PS)} \}$.

\begin{theo}{maintheorem}
Let $(\ct{E}, \ct{S})$ be a predicative category with small maps for which all the axioms in ${\cal A}$ hold and let $(\ct{C}, {\rm Cov})$ be an internal Grothendieck site in \ct{E}, such that the codomain map $\ct{C}_1 \to \ct{C}_0$ is small and a presentation for the topology exists. Then in the category of internal sheaves $\shct{\ct{E}}{\ct{C}}$  one can identify a class of maps making it into a predicative category with small maps for which the axioms in ${\cal A}$ holds as well. 
\end{theo}

In combination with \reftheo{existenceofinitZFalgebra} this result can be used to prove the existence of sheaf models of various constructive set theories:

\begin{coro}{maincorollary}
Suppose that $(\ct{E}, \ct{S})$ is a predicative category with small maps satisfying the axiom ${\bf (F)}$ and suppose that $(\ct{C}, {\rm Cov})$ is an internal Grothendieck site in \ct{E}, such that the codomain map $\ct{C}_1 \to \ct{C}_0$ is small and a presentation for the topology exists. Then the initial \spower-algebra in $\shct{\ct{E}}{\ct{C}}$ exists and is a model of {\bf CZF}. If, moreover,
\begin{enumerate}
\item the axioms ${\bf (AMC)}$ and ${\bf (WS)}$ hold in \ct{E}, then the initial \spower-algebra in $\shct{\ct{E}}{\ct{C}}$ also models Aczel's Regular Extension Axiom.
\item the axioms ${\bf (M)}$ and ${\bf (PS)}$ hold in \ct{E}, then the initial \spower-algebra in $\shct{\ct{E}}{\ct{C}}$ is a model of {\bf IZF}.
\end{enumerate}
\end{coro}

\section{Presheaves}

In this section we show that predicative categories with small maps are closed under presheaves. More precisely, we show that if $(\ct{E}, \smallmap{S})$ is a predicative category with small maps and \ct{C} is an internal category in \ct{E}, then inside the category \pshct{\ct{E}}{\ct{C}} of internal presheaves one can identify a class of maps such that \pshct{\ct{E}}{\ct{C}} becomes a predicative category with small maps. Our argument proceeds in two steps. First, we need to identify a suitable class of maps in a category of internal presheaves. We take what we will call the pointwise small maps of presheaves. To prove that these pointwise small maps satisfy axioms {\bf (A1-9)}, we need to assume that the codomain map of \ct{C} is small (note that the same assumption was made in \cite{warren07}). Subsequently, we show that the validity in the category with small maps $(\ct{E}, \smallmap{S})$ of any of the axioms introduced in the previous section implies its validity in any category of internal presheaves over $(\ct{E}, \smallmap{S})$. To avoid repeating the convoluted expression ``the validity of axiom {\bf (X)} in a predicative category with small maps implies its validity in any category of internal presheaves over it'', we will write ``{\bf (X)} is inherited by presheaf models'' or ``{\bf (X)} is stable under presheaf extensions'' to express this.

The main result of this section is that the fullness axiom {\bf (F)} is stable under presheaf extensions. Most of the other stability results in this section are not really new and can in one form or another already be found in \cite{joyalmoerdijk95, moerdijkpalmgren00, moerdijkpalmgren02, gambino05, warren07}. Nevertheless, for several reasons, we have decided to include their proofs here. First of all, none of the references we mentioned uses conditions on the ambient category which are exactly the same as ours (in particular, we assume only bounded exactness). Secondly, these papers use different definitions of the class of small maps in presheaves, which we will compare in Section 3.2 below. And, thirdly, including them will make our presentation self-contained.

\subsection{Pointwise small maps in presheaves}

Throughout this section, we work in a predicative category with small maps $(\ct{E}, \smallmap{S})$ in which we are given an internal category \ct{C}, whose codomain map 
\[ \func{\rm cod}{\ct{C}_1}{\ct{C}_0} \]
is small. Here we have written $\ct{C}_0$ for the object of objects of \ct{C} and $\ct{C}_1$ for its object of arrows. In addition, we will write \pshct{\ct{E}}{\ct{C}} for the category of internal presheaves, and $\pi^*$ for the forgetful functor:
\[ \func{\pi^*}{\pshct{\ct{E}}{\ct{C}}}{\ct{E}/\ct{C}_0}. \]
In the sequel, we will use capital letters for presheaves and morphisms of presheaves, and lower case letters for objects and morphisms in \ct{C}.

We will also employ the following piece of notation. For any map of presheaves $F: Y \to X$ and element $x \in X(a)$, we set
\[ Y_x^M := \{ \, (f \in \ct{C}_1, y \in Y({\rm dom} f)) \, : \, {\rm cod}(f) = a, F(y) = x \cdot f \, \}. \]
(The capital letter $M$ stands for the maximal sieve on $b$: for this reason, this piece of notation is consistent with the one to be introduced in Section 4.4.) Occasionally, we will regard $Y^M_x$ as a presheaf: in that case, its fibre at $b \in \ct{C}_0$ is
\[ Y_x^M(b) = \{ \, (f: b \to a \in \ct{C}_1, y \in Y(b)) \, : \, F_b(y) = x \cdot f \, \}, \]
and the restriction of an element $(f, y) \in Y_x^M(b)$ along $g: c \to b$ is given by
\[ (f, y) \cdot g = (fg, y \cdot g). \]

A map of presheaves $F: Y \to X$ will be called \emph{pointwise small}, if $\pi^* F$ belongs to $\smallmap{S}/\ct{C}_0$ in $\ct{E}/\ct{C}_0$. Note that for any such pointwise small map of presheaves and for any $x \in X(a)$ with $a \in \ct{C}_0$ the object $Y^M_x$ will be small. This is an immediate consequence of the fact that the codomain map is assumed to be small.

\begin{theo}{smallmapsinpresh}
The pointwise small maps make \pshct{\ct{E}}{\ct{C}} into a category with small maps. 
\end{theo}
\begin{proof} Observe that finite limits, images and sums of presheaves are computed ``pointwise'', that is, as in $\ct{E}/\ct{C}_0$. The universal quantification of $A \subseteq Y$ along \func{F}{Y}{X} is given by the following formula: for any $a \in \ct{C}_0$,
\begin{labequation}{forallinpresh}
\begin{array}{lcl}
\forall_F (A)(a) & = & \{ \, x \in X(a) \, : \, \forall (f, y) \in Y^M_x \, (  y \in A   ) \, \} 
\end{array}
\end{labequation}%
This shows that \pshct{\ct{E}}{\ct{C}} is a positive Heyting category. To complete the proof, we need to check that the pointwise small maps in presheaves satisfy axioms {\bf (A1-9)}. We postpone the proof of the collection axiom {\bf (A7)} (it will be \refprop{collinpresh}). The remaining axioms follow easily, as all we need to do is verify them pointwise. For verifying axiom {\bf (A8)}, one observes that the universal quantifier in \refeq{forallinpresh} ranges over a small object.
\end{proof}

For most of the axioms that we introduced in Section 2, it is relatively straightforward to check that they are inherited by presheaf models. The exceptions are the representability, collection and fullness axioms: verifying these requires an alternative characterisation of the small maps in presheaves and they will therefore be discussed in a separate section.

\begin{prop}{someaxiomsstunderpresh}
The following axioms are inherited by presheaf models: {\bf (M)}, bounded exactness, {\bf (NE)} and {\bf (NS)}, as well as {\bf ($\Pi$E)}, {\bf ($\Pi$S)} and {\bf (PS)}.
\end{prop}
\begin{proof}
The monomorphisms in presheaves are precisely those maps which are pointwise monic and therefore the axiom {\bf (M)} will be inherited by presheaf models. Similarly, presheaf models inherit bounded exactness, because quotients of equivalence relations are computed pointwise. Since the natural numbers objects in presheaves has that of the base category \ct{E} in every fibre, both {\bf (NE)} and {\bf (NS)} are inherited by presheaf models.

Finally, consider the following diagram in presheaves, in which $F$ is small:
\diag{ B \ar[d]_{G} & \\ Y \ar[r]_{F} & X. }
The object $P = \Pi_{F}(G)$ over an element $x \in X(a)$ is given by the formula:
\[
P_x := \{ \, s \in \Pi_{(f, y) \in Y^M_x} G^{-1} (y) \, : \, s \mbox{ is natural} \, \}. \]
This shows that {\bf ($\Pi$E)} is inherited by presheaf extensions. It also shows that {\bf ($\Pi$S)} is inherited, because the formula
\[ \forall (f, y) \in Y^M_x(b) \, \forall g: c \to b \, (  s(f, y) \cdot g = s(fg, y \cdot g)  ) \]
expressing the naturality of $s$ is bounded.

To see that {\bf (PS)} is inherited, we first need a description of the $\spower$-functor in the category of internal presheaves. This was first given by Gambino in \cite{gambino05} and works as follows. If $X$ is a presheaf and ${\bf y}c$ is the representable presheaf on $c \in \ct{C}_0$, then 
\[ \spower(X)(c) = \{ \, A \subseteq {\bf y}c \times X \, : \, A \mbox{ is a small subpresheaf}  \, \}, \]
with restriction along $f: d \to c$ on an element $A \in \spower(X)(c)$ defined by
\[ (A \cdot f)(e) = \{ (g: e \to d, x \in X(e)) \, : \, (fg, x) \in A \}. \]
The membership relation $\in_X \subseteq X \times \spower X$ is defined on an object $c \in \ct{C}$ by: for all $x \in X(c)$ and $A \in \spower(X)(c)$,
\[ x \in_X A \Longleftrightarrow (\id_c, x) \in A. \]
This shows that the axiom {\bf (PS)} is inherited, because the formula
\[ \forall (f: b \to c, x) \in A \,  \forall g: a \to b \, [ \, (fg, x \cdot g) \in A \, ] \]
expressing that $A$ is a subpresheaf is bounded.
\end{proof}

\begin{theo}{Winpresh}
The axioms {\bf (WE)} and {\bf (WS)} are inherited by presheaf extensions.
\end{theo}
\begin{proof}
For this proof we need to recall the construction of polynomial functors and W-types in presheaves from \cite{moerdijkpalmgren00}. For a morphism of presheaves $F: Y \to X$ and a presheaf $Z$, the value of
\[ P_{F}(Z) = \sum_{x \in X} Z^{Y_x} \]
on an object $a$ of $\ct{C}_0$ is given by
\[ P_F(Z)(a) = \{ \, (x \in X(a), t: Y^M_x \to Z) \}, \]
where $t$ is supposed to be a morphism of presheaves. The restriction of an element $(x, t)$ along a map $f: b \to a$ is given by $(x \cdot f, f^*t)$, where
\[ (f^*t)(g, y) = t(fg, y). \]

The presheaf morphism $F$ induces a map 
\[ \func{\phi}{\sum_{a \in \ct{C}_0} \sum_{x  \in X(a)} Y^M_x }{\sum_{a \in \ct{C}_0}  X(a)} \]
in \ct{E} whose fibre over $x \in X(a)$ is $Y^M_x$ and which is therefore small. The W-type in presheaves will be constructed from the W-type $V$ associated to $\phi$ in \ct{E}.

A typical element $v \in V$ is a tree of the form
\[ v = {\rm sup}_{x} t \]
where $x$ is an element of some $X(a)$ and $t$ is a function $Y^M_x \to V$. For any such $v$, one defines its root $\rho(v)$ to be $a$. If one writes $V(a)$ for the set of trees $v$ such that $\rho(v) = a$, the object $V$ will carry the structure of a presheaf, with the restriction of an element $v \in V(a)$ along a map \func{f}{b}{a} given by
\[ v \cdot f = {\rm sup}_{x \cdot f} f^{*}t. \]

The W-type associated to $F$ in presheaves is obtained by selecting the right trees from $V$, the right trees being those all whose subtrees are (in the terminology of \cite{moerdijkpalmgren00}) composable and natural. A tree $v = {\rm sup}_x(t)$ is called \emph{composable} if for all $(f, y) \in Y^M_x$,
\[ \rho(t(f, y)) = \mbox{dom}(f). \]
A tree $v = {\rm sup}_x(t)$ is \emph{natural}, if it is composable and for any $(f, y) \in Y^M_x(a)$ and any $g: b \to a$, we have
\[ t(f, y) \cdot g = t(fg, y \cdot g) \]
(so $t$ is actually a natural transformation). A tree will be called \emph{hereditarily natural}, if all its subtrees (including the tree itself) are natural.

In \cite[Lemma 5.5]{moerdijkpalmgren00} it was shown that for any hereditarily natural tree $v$ rooted in $a$ and map $f: b \to a$ in \ct{C}, the tree $v \cdot f$ is also hereditarily natural. So when $W(a) \subseteq V(a)$ is the collection of hereditarily natural trees rooted in $a$, $W$ is a subpresheaf of $V$.

A proof that $W$ is the W-type for $F$ can be found in the sources mentioned above. Presently, the crucial point is that the construction can be imitated in our setting, so that {\bf (WE)} is stable under presheaves. The same applies to {\bf (WS)}, essentially because $W$ was obtained from $V$ using bounded separation (in this connection it is essential that the object of all subtrees of a particular tree $v$ is small, see \cite[Theorem 6.13]{bergmoerdijk08}).
\end{proof}

\subsection{Locally small maps in presheaves}

For showing that the representability, collection and fullness axioms are inherited by presheaf models, we use a different characterisation of the small maps in presheaves: we introduce the \emph{locally small maps} and show that these coincide with the pointwise small maps. To define these locally small maps, we have to set up some notation.

\begin{rema}{leftadjoint}
The functor  \func{\pi^*}{\pshct{\ct{E}}{\ct{C}}}{\ct{E}/\ct{C}_0} has a left adjoint, which is computed as follows: to any object $(X, \sigma_X: X \to \ct{C}_0)$ and $a \in \ct{C}_0$ one associates
\[ \pi_!(X)(a) = \{ (x \in X, f: a \to b)  \, : \, \sigma_X(x) = b \}, \]
which is a presheaf with restriction given by
\[ (x, f) \cdot g = (x, fg). \]
This means that $\pi^* \pi_! X$ fits into the pullback square
\diag{ \pi^* \pi_! X \ar[d] \ar[r] & \ct{C}_1 \ar[d]^{\rm cod} \\
X \ar[r]_{\sigma_X} & \ct{C}_0. }
From this one immediately sees that $\pi_!$ preserves smallness. Furthermore, the component maps of the counit $\pi_! \pi^* \to 1$ are small covers (they are covers, because under $\pi^*$ they become split epis in $\ct{E}/\ct{C}_0$; that they are also small is another consequence of the fact that the codomain map is assumed to be small).
\end{rema}

In what follows, natural transformations of the form 
\[ \pi_! B \to \pi_! A \]
will play a crucial r\^{o}le and therefore it will be worthwhile to analyse them more closely. First, due to the adjunction, they correspond to maps in $\ct{E}/\ct{C}_0$ of the form
\[ B \to \pi^* \pi_! A. \]
Such a map is determined by two pieces of data: a map $r: B \to A$ in $\ct{E}$, and, for any $b \in B$, a morphism $s_b: \sigma_B(b) \to \sigma_A(rb)$ in \ct{C}, as depicted in the following diagram:
\diaglab{defshriek}{
B \ar[d]_r \ar[r]_s \ar@/^1pc/[rr]^{\sigma_B} & \ct{C}_1 \ar[d]^{\rm cod} \ar[r]_{\rm dom} & \ct{C}_0 \\
A \ar[r]_{\sigma_A} & \ct{C}_0.}
(Note that we do not have  $\sigma_A r = \sigma_B$ in general, so that it is best to consider $r$ as a map in \ct{E}.) We will use the expression $(r, s)$ for the map $B \to \pi^* \pi_! A$ and $(r, s)_!$ for the natural transformation $\pi_! B \to \pi_! A$ determined by a diagram as in \refdiag{defshriek}.

In the following lemma, we collect the important properties of the operation $(-,-)_!$.
\begin{lemm}{propofshriek}
\begin{enumerate}
\item Assume $r$ and $s$ are as in diagram \refdiag{defshriek}. Then $(r, s)_!: \pi_! B \to \pi_! A$ is a pointwise small map of presheaves iff $r: B \to A$ is small in $\ct{E}$.
\item Assume $r: B \to A$ is a cover and $\sigma_A: A \to \ct{C}_0$ is an arbitrary map. If we set $\sigma_B = \sigma_A r$ and $s_b = \id_{\sigma_B b}$ for every $b \in B$, then $(r, s)_!: \pi_! B \to \pi_!A$ is a cover.
\item If $(r, s): B \to \pi^* \pi_! A$ is a cover and $\sigma_B(b) = {\rm dom}(s_b)$ for all $b \in B$, then also $(r, s)_!: \pi_! B \to \pi_! A$ is a cover.
\item If $(r, s)_!: \pi_!B \to \pi_! A$ is a natural transformation determined by a diagram as in \refdiag{defshriek} and we are given a commuting diagram 
\diag{  V \ar[d]_h \ar[r]^p & B \ar[d]^{r} \\
 W \ar[r]_q & A}
in $\ct{E}$, then these data induce a commuting square of presheaves
\diag{  \pi_! V \ar[r]^{\pi_! p} \ar[d]_{(h, sp)_!} & \pi_! B \ar[d]^{(r, s)_!} \\
\pi_! W \ar[r]_{\pi_! q} & \pi_! A,}
with $\sigma_V = \sigma_B p$ and $\sigma_W = \sigma_A q$. Moreover, if the original diagram is a pullback (resp.~a quasi-pullback or a covering square), then so is the induced diagram. 
\item If $(r, s)_!: \pi_! A \to \pi_! X$ and $(u, v)_!: \pi_! B \to \pi_! X$ are natural transformations with the same codomain and for every $x \in X$ and every pair $(a, b) \in A \times_X B$ with $x = ra = ub$ there is a pullback square
\diag{ k_{(a, b)} \ar[d]_{p_{(a, b)}} \ar[r]^{q_{(a, b)}} & \sigma_B(b) \ar[d]^{v_b} \\
\sigma_A(a) \ar[r]_{s_a} & \sigma_X(x) }
in $\ct{C}$, then $\pi_!$ applied to the object $\sigma_{A \times_X B}: A \times_X B \to \ct{C}_0$ in $\ct{E}/\ct{C}_0$ obtained by sending $(a, b) \in A \times_X B$ to $k_{(a,b)}$ is the pullback of $(r, s)_!$ along $(u, v)_!$ in $\pshct{\ct{E}}{\ct{C}}$:
\diag{ \pi_!(A \times_X B)  \ar[r]^(.6){(\pi_2, q)_!} \ar[d]_{(\pi_1, p)_!} & \pi_! B \ar[d]^{(u, v)_!} \\
\pi_! A \ar[r]_{(r, s)_!} & \pi_!X. }
\end{enumerate}
\end{lemm}
\begin{proof}
By direct inspection.
\end{proof}

Using the notation we have set up, we can list the two notions of a small map of presheaves.
\begin{enumerate}
\item The pointwise definition (as in the previous section): a map \func{F}{B}{A} of presheaves is \emph{pointwise small}, when $\pi^* F$ is a small map in $\ct{E}/\ct{C}_0$.
\item The local definition (as in \cite{joyalmoerdijk95}): a map \func{F}{B}{A} of presheaves is \emph{locally small}, when $F$ is covered by a map of the form $(r, s)_!$ in which $r$ is small in \ct{E}.
\end{enumerate}

We show that these two classes of maps coincide, so that henceforth we can use the phrase ``small map'' without any danger of ambiguity.
\begin{prop}{threecoincide}
A map is pointwise small iff it is locally small.
\end{prop}
\begin{proof}
We have already observed that maps of the form $(r, s)_!$ with $r$ small are pointwise small, so all maps covered by one of this form are pointwise small as well. This shows that locally small maps are pointwise small. That all pointwise small maps are also locally small follows from the next lemma and the fact that the counit maps $\pi_! \pi^* Y \to Y$ are covers.
\end{proof}

\begin{lemm}{thirdchar}
For any pointwise small map \func{F}{Z}{Y} and any map \func{L}{\pi_! B}{Y} there is a quasi-pullback square of presheaves of the form
\diag{\pi_! C \ar[d]_{(k, l)_!} \ar[r] & Z \ar[d]^F \\
\pi_! B \ar[r]_L & Y, }
with $k$ small in \ct{E}.
\end{lemm}
\begin{proof}
Let $S$ be the pullback of $F$ along $L$ and cover $S$ using the counit as in:
\diag{
\pi_! \pi^* S \ar@{->>}[r] & S \ar[r] \ar[d] & \pi_! B  \ar[d]^L  \\
&  Z \ar[r]_{F} & Y.}
We know the composite along the top is of the form $(k, l)_!$. Because $k$ is the composite along the middle of the following diagram and both squares in this diagram are pullbacks, $k$ is the composite of two small maps and hence small.
\diag{ & \ct{C}_1 \ar[r]^{\rm cod} & \ct{C}_0 \\
\pi^* S \ar[r] \ar[d] & \pi^* \pi_! B \ar[u] \ar[d]^{\pi^* L} \ar[r] & B \ar[u] \\
\pi^* Z \ar[r]_{\pi^* F} & \pi^* Y }
\end{proof}

\begin{coro}{pointwisecoveredbylocallysmall}
Every pointwise small map is covered by one of the form $(r, s)_!$ in which $r$ is small. In fact, every composable pair $(G, F)$ of pointwise small maps of presheaves fits into a double covering square of the form
\diag{ \pi_! C \ar[d]_{(k, l)_!} \ar@{->>}[r] & Z \ar[d]^G \\
\pi_! B \ar[d]_{(r, s)_!}  \ar@{->>}[r] & Y \ar[d]^F \\
\pi_! A \ar@{->>}[r] & X, }
in which $k$ and $r$ are small in \ct{E}.
\end{coro}
\begin{proof}
We have just shown that every pointwise small map is covered by one of the form $(r, s)_!$ in which $r$ is small, which is the first statement. The second statement follows immediately from this and the previous lemma.
\end{proof}

Using this alternative characterisation, we can quickly show that the collection axiom is inherited by presheaf models, as promised.
\begin{prop}{collinpresh}
The collection axiom {\bf (A7)} is inherited by presheaf models.
\end{prop}
\begin{proof}
Let \func{F}{M}{N} be a small map and \func{Q}{E}{M} be a cover. Without loss of generality, we may assume that $F$ is of the form $(k, l)_!$ for some small map $\func{k}{X}{Y}$ in \ct{E}. 

Let $n$ be the map obtained by pullback in $\ct{E}/\ct{C}_0$:
\diag{ T \ar@{->>}[r]^n \ar[d] & X \ar[d] \\
\pi^* E \ar@{->>}[r]_{\pi^* Q} & \pi^* \pi_! X.}
Then use collection in \ct{E} to obtain a covering square as follows:
\diag{ B \ar[r]^m \ar[d]_d & T \ar@{->>}[r]^n & X \ar[d]^k  \\ 
A \ar@{->>}[rr]_p & & Y. }
Using \reflemm{propofshriek}.4 this leads to a covering square in the category of presheaves
\diag{ \pi_! B \ar[r]^{\pi_! m} \ar[d]_{(d, lnm)_!} & \pi_! T \ar[r] \ar@/^1pc/[rr]^{\pi_! n}  & E \ar@{->>}[r]_{Q} & \pi_! X \ar[d]^{(k, l)_!}  \\ 
\pi_! A \ar@{->>}[rrr]_{\pi_!p} & & & \pi_! Y, }
thus completing the proof.
\end{proof}

\begin{prop}{Rinpresh}
The representability axiom is inherited by presheaf models.
\end{prop}
\begin{proof}
Let \func{\pi}{E}{U} be a universal small map in \ct{E}, and define the following two objects in $\ct{E}/\ct{C}_0$:
\begin{eqnarray*}
U' & = & \{ (u \in U, c \in \ct{C}_0, p: E_u \to \ct{C}_1) \, : \, \forall e \in E_u \, ({\rm cod}(pe) = c) \}, \\
\sigma_{U'}(u, c, p) & = & c, \\
E' & = & \{ (u, c, p, e) \, : \, (u, c, p) \in U', e \in E_u \}, \\
\sigma_{E'}(u, c, p, e) & = & {\rm dom}(pe). \\
\end{eqnarray*}
If $r: E' \to U'$ is the obvious projection and $s: E' \to \ct{C}_1$ is the map sending $(u, c, p, e)$ to $pe$, then $r$ and $s$ fit into a commuting square as shown:
\diag{ E' \ar[d]_{r} \ar[r]_s \ar@/^1pc/[rr]^{\sigma_{E'}} & \ct{C}_1 \ar[d]^{\rm cod} \ar[r]_{\rm dom} & \ct{C}_0 \\
U' \ar[r]_{\sigma_{U'}} & \ct{C}_0.}
We claim that the induced map $(r, s)_!$ in the category of presheaves is a universal small map. To show this, we need to prove that any small map $F$ can be covered by a pullback of $(r, s)_!$. Without loss of generality, we may assume that $F = (k, l)_!$ for some small map \func{k}{X}{Y} in $\ct{E}$.

Since $\pi$ is a universal small map, there exists a diagram of the form
\diag{ E \ar[d]_{\pi} & V \ar[l]_m \ar[d]_h \ar[r]^i & X \ar[d]^{k}  \ar[r]^l & \ct{C}_1 \ar[d]^{\rm cod} \\
U &  W \ar[l]^n \ar[r]_j & Y \ar[r]_{\sigma_Y} & \ct{C}_0, }
in which the left square is a pullback and the middle one a covering square. From this, we obtain a commuting diagram of the form
\diag{
& \ct{C}_0 \\
 V \ar[d]_h \ar[r]_{m'} \ar@/^1pc/[rr]^{li} \ar[ru]^{\sigma_V} & E' \ar[d]^{r} \ar[r]_s & \ct{C}_1 \ar[d]^{\rm cod} \ar[ul]_{\rm dom} \\
W \ar[r]^{n'} \ar@/_1pc/[rr]_{\sigma_W} & U' \ar[r]^{\sigma_U} & \ct{C}_0}
by putting
\begin{eqnarray*}
\sigma_W & = & \sigma_Y j, \\
n'w & = & (nw, \sigma_Ww, \lambda e \in E_{nw}. l_{im^{-1}e}), \\
\sigma_V & = & \sigma_X i,  \\
m'v & = & (n'hv, mv).
\end{eqnarray*}
Together these two commuting diagrams determine a diagram in the category of internal presheaves
\diag{ \pi_! E' \ar[d]_{(r, s)_!} & \pi_! V \ar[l]_{\pi_! m'} \ar[d]^{(h, li)_!} \ar[r]^{\pi_! i} & \pi_! X \ar[d]^{(k, l)_!}  \\
\pi_! U' & \pi_! W \ar[l]^{\pi_! n'} \ar[r]_{\pi_! j} & \pi_! Y, }
in which the left square is a pullback and the right one a covering square (by \reflemm{propofshriek}.4).
\end{proof}

\begin{theo}{stabFforpresh}
(Assuming \ct{C} has chosen pullbacks.) The fullness axiom {\bf (F)} is inherited by presheaf models.
\end{theo}
\begin{proof}
In view of \reflemm{redlemmforfull} and \refcoro{pointwisecoveredbylocallysmall}, we only need to build generic \emph{mvs}s for maps of the form $(k, l)_!: \pi_! B \to \pi_! A$ in which $k$ is small, where $\pi_! A$ lies over some object of the form $\pi_! X$ via a map of the form $(r, s)_!$ in which $r$ is small. To construct this generic \emph{mvs}, we have to apply fullness in \ct{E}. For this purpose, consider the object
\begin{eqnarray*}
B_0 & = & \{ (b \in B, f \in \ct{C}_1, g \in \ct{C}_1) \, : \, \sigma_X(rkb) = {\rm cod}(f), (f^*l_b) g = \id \}.
\end{eqnarray*}
Here $f^*l_b$ is understood to be the map fitting, for any $b \in B$ and $f: d \to c$ with $c = \sigma_X(rkb)$, in the double pullback diagram
\diag{ f^* \sigma_B(b) \ar[r] \ar[d]_{f^*l_b} & \sigma_B(b) \ar[d]^{l_b} \\
f^* \sigma_A(kb) \ar[r] \ar[d]_{f^*s_{kb}} & \sigma_A(kb) \ar[d]^{s_{kb}} \\ 
d \ar[r]_f & c }
in $\ct{C}$. If we write $k_0: B_0 \to A$ for the map sending $(b, f, g)$ to $k(b)$, then this map is small, so we can use fullness in \ct{E} to find a cover $n: W \to X$ and a small map $m_0: Z_0 \rTo W$, together with a generic \emph{mvs} $P_0$ for $k_0$ over $Z_0$, as depicted in the following diagram.
\diag{ P_0 \ar@{>->}[r] \ar@{->>}@/_/[dr]
 & Z_0 \times_{X} B_0 \ar[d] \ar[rr] & & B_0 \ar[d]^{k_0} \\
& Z_0 \times_{X} A \ar[rr] \ar[d] & & A \ar[d]^{r} \\
& Z_0 \ar[r]_{m_0}  & W \ar[r]_n & X }

Now we make a number of definitions:
\begin{eqnarray*}
Z & = & \{ \, (z_0 \in Z_0, f \in \ct{C}_1) \, : \, {\rm cod}(f) = nm_0(z_0) \mbox{ and } \\
&  & (\forall a \in A_{nm_0(z_0)}) \, (\exists b \in B_a) \, (\exists g \in \ct{C}_1) \, (z_0, b, f, g) \in P_0 \, \}, \\
\sigma_Z(z_0, f) & = & {\rm dom}(f), \\
\mu(z_0, f) & = & f, \\
\sigma_W(w) & = & \sigma_X(nw).
\end{eqnarray*}
In addition, we write $m_1: Z \to Z_0$ for the obvious (small) projection and $m = m_0m_1$. Then we obtain the following diagram of presheaves, in which both rectangles are pullbacks computed using \reflemm{propofshriek}.5:
\diag{ \pi_!(Z \times_X B) \ar[d] \ar[rr] & & \pi_! B \ar[d]^{(k, l)_!} \\
\pi_! (Z \times_X A) \ar[d] \ar[rr] & & \pi_! A \ar[d]^{(r, s)_!}  \\
 \pi_! Z \ar[r]_{(m, \mu)_!} & \pi_! W \ar[r]_{\pi_!n} & \pi_! X. }
We wish to define a subpresheaf of $\pi_!(Z \times_X B)$ and prove that it is the generic \emph{mvs} of $(k, l)_!$. We can do this by saying:
\begin{quote}
$(z_0, f, b, h) \in P$ if $h$ factors through a map $g$ with $(z_0, b, f, g) \in P_0$.
\end{quote}
The inclusion of $P$ in $\pi_!(Z \times_X B)$ is bounded, because $P$ is defined by a bounded formula (using that the codomain map is small). Furthermore, the induced map from $P$ to $\pi_!(Z \times_X A)$ is a cover by definition of $Z$. Thus it remains to verify genericity.

To verify this, let $E \to \pi_! W$ be any map and $Q$ be an \emph{mvs} of $(k, l)_!$ over $E$. Without loss of generality, we may assume that $E$ is of the form $\pi_! Y$ (since $E$ can always be covered using the counit). This leads to the following diagram of presheaves in which the rectangles are pullbacks:
\diag{ 
Q \ar@{>->}[r] \ar@{->>}@/_/[rd] & \bullet \ar[d] \ar[rr] & & \pi_! B \ar[d]^{(k, l)_!}  \\
 & \bullet \ar[d] \ar[rr] & & \pi_! A  \ar[d]^{(r, s)_!} \\
 & \pi_! Y  \ar[r]_{(d, \delta)_!}  &  \pi_! W \ar[r]_{\pi_!n} & \pi_! X. }
Of course, we will assume that the pullbacks are computed using \reflemm{propofshriek}.5, so that they are $\pi_!(Y \times_X B)$ and $\pi_!(Y \times_X A)$, respectively. It follows that in \ct{E} we have an \emph{mvs} $Q_0$ for $k_0$ over $Y$, as in
\diag{ Q_0 \ar@{>->}[r] \ar@{->>}@/_/[dr]
 & Y \times_{X} B_0 \ar[d] \ar[rr] & & B_0 \ar[d]^{k_0} \\
& Y \times_{X} A \ar[rr] \ar[d] & & A \ar[d]^{r} \\
& Y \ar[r]_{d}  & W \ar[r]_n & X, }
given by
\[ (y, b, f, g) \in Q_0 \mbox{ if }  f = \delta_y \mbox{ and } (y, b, g) \in Q. \]
Therefore, by the genericity of $P_0$, there is a cover  \func{e}{U}{Y} and a map  \func{c}{U}{Z_0} with $de = m_0 c$ and
\begin{equation} \label{inclofmvss}
c^*P_0 \subseteq e^*Q_0.
\end{equation}

\emph{Claim:} If we put $t(u) = (c(u), \delta_{e(u)})$, then $t(u) \in Z$ for every $u \in U$. \emph{Proof:} Suppose $a \in A_{nm_0c(u)}$. We know that there are $b \in B_a, f, g \in \ct{C}_1$ such that $(c(u), b, f, g) \in P_0$, because $P_0 \to Z_0 \times_X A$ is surjective, but the question is: do we have $f = \delta_{e(u)}$? The answer is \emph{yes}, because if $(c(u), b, f, g) \in P_0$, then $(e(u), b, f, g) \in Q_0$ by (\ref{inclofmvss}). So we have $f = \delta_{e(u)}$ by definition of $Q_0$. $\Box$

It follows that if we put 
\[ \sigma_U(u) := {\rm dom}(\delta_{e(u)}) = \sigma_Y(e(u)) = \sigma_Z(t(u)), \]
then we have the following diagram of presheaves:
\diag{ \pi_!U \ar[r]^{\pi_!t} \ar@{->>}[d]_{\pi_!e} & \pi_!(Z) \ar[d]^{(m, \mu)_!} \\
\pi_!Y \ar[r]_{(d, \delta)_!} & \pi_!W.}
To see that this square commutes, we need to chase an element from $\pi_!U$ along the two sides and it suffices to this for an element of the form $(u, \id)$.
\begin{eqnarray*}
(m, \mu)_!\pi_!t(u, \id) & = & (m, \mu)_!(t(u), \id) \\
& = & (m, \mu)_!(c(u), \delta_{e(u)}, \id) \\
& = & (m_0c(u), \delta_{e(u)}) \\
& = & (de(u), \delta_{e(u)}) \\
& = & (d, \delta)_!(e(u), \id) \\
& = & (d, \delta)_!\pi_!e(u, \id). 
\end{eqnarray*}
Therefore the proof will be finished, once we show that $(\pi_!t)^* P \subseteq (\pi_!e)^* Q$. 

To show this, consider an element $(u, b, h) \in (\pi_!t)^*P$. We then have $(t(u), b, h) \in P$, which, by definition of $P$, means that $h$ factors through a map $g$ such that $(c(u), b, \delta_{e(u)}, g) \in P_0$. From (\ref{inclofmvss}) it follows that $(e(u), b, \delta_{e(u)}, g) \in Q_0$ and hence $(e(u), b, g) \in Q$, by definition of $Q_0$. Since $Q$ is a presheaf, we also have $(e(u), b, h) \in Q$, whence $(u, b, h) \in (\pi_!e)^*Q$, as desired.
\end{proof}

\begin{rema}{fullnesspreshinternal}
Diagrammatic proofs as the one we just gave are hard to read and motivate. One can give a more understandable proof using the internal logic of the category of presheaves: for those who are familiar with its intricacies, we present such a proof below.

\begin{theo}{fullnesspreshinternallogic}
(Assuming \ct{C} has chosen finite products.) The fullness axiom {\bf (F)} is inherited by presheaf models.
\end{theo}
\begin{proof}
In view of \reflemm{redlemmforfull} and \refcoro{pointwisecoveredbylocallysmall} we only need to find generic \emph{mvs}s for small map $ (k, l)_!: \pi_!B \to \pi_!A$, where $\pi_!A$ is fibred by a small map $(r, s)_!: \pi_! A \to \pi_!X$ over $\pi_!X$. Then, by replacing $\ct{C}$ by $\ct{C}/ \pi_!X = \sum_{x \in X} \ct{C}/ \sigma_X(x)$, we may even assume that $X = 1 = \{ * \}$ and $\sigma_X(*) = 1$.  

Internal universal quantification over \emph{mvs}s of $\pi_!B$ in the category $\pshct{\ct{E}}{\ct{C}}$ amounts to $\ct{E}$-internal quantification over certain subpresheaves $P$ of $\pi_!(B) \times \ct{C}(-,c)$, namely those which are \emph{mvs}s over $\pi_!(A) \times \ct{C}(-,c)$. Such a subpresheaf satisfies
\begin{displaymath}
\begin{array}{c}
 (\forall a \in A) \, (\forall h: d \to \sigma_A(a)) \, (\forall f: d \to c) \, (\exists b \in B) \, (\exists g: d \to \sigma_B(b)) \\  \, (k(b), l_b \circ g) = (a, h) \mbox{ and } (b, g, f) \in P \subseteq \pi_!(B) \times \ct{C}(-,c) \mbox{ at } d,
\end{array}
\end{displaymath}
which is equivalent to
\begin{displaymath}
\begin{array}{c}
 (\forall a \in A) \, (\exists b \in B) \, (\exists g: \sigma_A(a) \times c \to \sigma_B(b)) \\  \, (k(b), l_b \circ g) = (a, \pi_1) \mbox{ and } (b, g, \pi_2) \in P \subseteq \pi_!(B) \times \ct{C}(-,c) \mbox{ at } \sigma_A(a) \times c,
\end{array}
\end{displaymath}
or
\begin{equation}  \label{mvsinpreschoven}
\begin{array}{c}
 (\forall a \in A) \, (\exists b \in B) \, (\exists g: \sigma_A(a) \times c \to \sigma_B(b)) \\  \, k(b) = a, l_b \circ g = \pi_1 \mbox{ and } (b, g, \pi_2) \in P(\sigma_A(a) \times c).
\end{array}
\end{equation}
We use fullness in \ct{E} to obtain a small family of subobjects $\{ Q_i \, : \, i \in I \}$ which form a generic family of \emph{mvs}s for
\[ \{ (b \in B, c \in \ct{C}_0, g: \sigma_A(kb) \times c \to \sigma_B(b)) \, : \, l_b \circ g = \pi_1 \} \to A: (b, c, g) \mapsto k(b). \]
From these $Q_i$ we now construct an internal small family of small presheaves of $\pi_!(B)$. Such a family is generated by subpresheaves of $\pi_!(B) \times \ct{C}(-, c)$ for varying $c \in \ct{C}_0$. For any such $c \in \ct{C}_0$, we take the presheaves
\begin{eqnarray*}
\hat{Q}_i^{(c)}(d)  & = & \{ (b \in B, gh, \pi_2h) \, : \,  g: \sigma_A(kb) \times c \to \sigma_B(b), h : d \to \sigma_A(kb) \times c \\
& & \mbox{ and } (b, c, g) \in Q_i \},
\end{eqnarray*}
provided $i$ and $c$ make the map $\hat{Q}_i^{(c)} \to \pi_!(A) \times \ct{C}(-, c)$ surjective.

Now we show these are generic. Take $c_0 \in \ct{C}$ and $P$ a \emph{mvs} over $c_0$, as in
\diag{ P \ar@{ >->}[r] \ar@{->>}[dr] & \pi_!(B) \times \ct{C}(-, c_0) \ar[d]^{(k, l)_! \times \id} \\
& \pi_!(A) \times \ct{C}(-, c_0). }
This $P$ satisfies (\ref{mvsinpreschoven}) for $c_0$, so there is a $Q_i$ contained in
\[ \{ (b \in B, c_0, g: \sigma_A(kb) \times c_0 \to \sigma_B(b)) \, : \, (b, g, \pi_2) \in P, l_b \circ g = \pi_1 \}. \]
We claim that the map $\hat{Q}_i^{(c_0)} \to \pi_!(A) \times \ct{C}(-, c_0)$ is surjective, and to show this it suffices to prove that elements of the form $(a, \pi_1: \sigma_A(a) \times c_0 \to c_0, \pi_2: \sigma_A(a) \times c_0 \to \sigma_A(a))$ are hit by this map. Since $Q_i$ is a \emph{mvs}, we know that there are $(b, c, g) \in Q_i$ with $k(b) = a$ and $l_b \circ g = \pi_1$. Since $Q_i \subseteq P$, we have $c = c_0$ and hence $(b, g, \pi_2) \in \hat{Q}_i^{(c_0)}$ and $((k, l)_! \times \id)(-, c_0)(b, g, \pi_2) = (a, \pi_1, \pi_2)$.

So it remains to check $\hat{Q}_i^{(c_0)} \subseteq P$. But if $(b, gh, \pi_2h) \in Q_i^{(c_0)}(d)$ for some maps $h: d \to \sigma_A(kb) \times c$ and $g: \sigma_A(kb) \times c \to \sigma_B(b)$ with $(b, c_0, g) \in Q_i$, then $(b, g, \pi_2) \in P(\sigma_A(kb) \times c)$ and hence $(b, gh, \pi_2h) = (b, g, \pi_2) \cdot h \in P(d)$.
\end{proof}

\end{rema}

\section{Sheaves}

In this section we continue to work in the setting of a predicative category with small maps $(\ct{E}, \smallmap{S})$ together with an internal category \ct{C} in \ct{E} whose codomain map is small. To define a category of internal sheaves, we have to assume that the category \ct{C} comes equipped with a Grothendieck topology, so as to become a Grothendieck site. There are different formulations of the notion of a site, all essentially equivalent (\cite{johnstone02b} provides an excellent discussion of this point), but for our purposes we find the following (``sifted'') formulation the most useful. 

\begin{defi}{site}
Let \ct{C} be an internal category whose codomain map in small. A \emph{sieve} $S$ on an object $a \in \ct{C}_0$ is a \emph{small} collection of arrows in \ct{C} all having codomain $a$ and closed under precomposition (i.e., if $f: b \to a$ and $g: c \to b$ are arrows in \ct{C} and $f$ belongs to $S$, then so does $fg$). Since we insist that sieves are small, there is an object of sieves (a subobject of $\spower \ct{C}_1$).

We call the set $M_a$ of all arrows into $a$ the \emph{maximal sieve} on $a$ (it is a sieve, since we are assuming that the codomain map is small). If $S$ is a sieve on $a$ and $f: b \to a$ is any map in \ct{C}, we write $f^*S$ for the sieve $\{ g: c \to b \, : \, fg \in S \}$ on $b$. In case $f$ belongs to $S$, we have $f^*S = M_b$.

A \emph{(Grothendieck) topology} Cov on \ct{C} is given by assigning to every object $a \in \ct{C}$ a collection of sieves ${\rm Cov}(a)$ such that the following axioms are satisfied:
\begin{description}
\item[(Maximality)] The maximal sieve $M_a$ belongs to ${\rm Cov}(a)$;
\item[(Stability)] If $f: b \to a$ is any map and $S$ belongs to ${\rm Cov}(a)$, then $f^*S$ belongs to ${\rm Cov}(b)$;
\item[(Local character)] If $S$ is a sieve on $a$ and $R \in {\rm Cov}(a)$ is such that for all $f: b \to a \in R$ the sieve $f^*S$ belongs to ${\rm Cov}(b)$, then $S$ belongs to ${\rm Cov}(a)$. 
\end{description}
A pair $(\ct{C}, {\rm Cov})$ consisting of a category $\ct{C}$ and a topology Cov on it is called a \emph{site}. If a site $(\ct{C}, {\rm Cov})$ has been fixed, we call the sieves belonging to some ${\rm Cov}(a)$ \emph{covering sieves}. If $S$ belongs to ${\rm Cov}(a)$ we say that $S$ is a \emph{sieve covering} $a$, or that $a$ \emph{is covered by} $S$.

Finally, a \emph{presentation} for a site $(\ct{C}, {\rm Cov})$ is a function BCov which yields, for every $a \in \ct{C}_0$, a \emph{small} collection of \emph{basic covering sieves} BCov($a$) such that:
\[ S \in {\rm Cov}(a) \Leftrightarrow \exists R \in {\rm BCov}(a): R \subseteq S. \] 
A site for which such a presentation exists will be called \emph{presentable}.\footnote{This is supposed to be reminiscent of Aczel's notion of a set-presentable formal space (see \cite{aczel06}). Note that in {\bf IZF} every site is presentable.}
\end{defi}

Our first goal in this section is prove that any category of internal sheaves over a predicative category with small maps $(\ct{E}, \smallmap{S})$ is a positive Heyting category. The proof of this relies on the existence of a sheafification functor (a left adjoint to the inclusion of sheaves in presheaves), and since this functor is built by taking a quotient, we use the bounded exactness of $(\ct{E}, \smallmap{S})$. To ensure that the equivalence relation by which we quotient is bounded, we will have to assume that the site is presentable. Next, we have to identify a class of small maps in any category of internal sheaves over $(\ct{E}, \smallmap{S})$. We will define pointwise small and locally small maps of sheaves and we will insist that these should again coincide (as happened in presheaves). For this to work out, we again seem to need the assumption that the site is presentable; moreover, we will assume that the fullness axiom holds in \ct{E} (note that similar assumptions were made in \cite{grayson83}). So, in effect, we will work in a predicative category with small maps $(\ct{E}, \smallmap{S})$ equipped with a Grothendieck site $(\ct{C}, {\rm Cov})$ such that:
\begin{enumerate}
\item The fullness axiom {\bf (F)} holds in \ct{E}.
\item The codomain map ${\rm cod}: \ct{C}_1 \to \ct{C}_0$ is small.
\item The site is presentable.
\end{enumerate}

After we have shown that a category of sheaves can be given the structure of a category with small maps, we prove that the validity of any of the axioms introduced in Section 2 in $(\ct{E}, \smallmap{S})$ implies its validity in any category of internal sheaves over it (Theorems 4.8--4.11 and Theorem 4.17): we will say that the axiom is ``inherited by sheaf models''. There is one exception to this, however: we will not be able to show that the axiom {\bf (WS)} is inherited by sheaf models. We will discuss the problem and provide a solution based on the axiom of multiple choice in Section 4.4 below (see \reftheo{stabilityofAMCundersheaves} and \reftheo{stabilityofWSundersheaves}). 

The main results of this section are that we establish the stability of fullness {\bf (F)} under sheaves and we correct the treatment of W-types in \cite{moerdijkpalmgren02}. In addition, we show that the two different notions of a class of small maps that occur in the literature coincide in our setting. As far as the basic axioms are concerned, their stability can in one form or another already be found in the literature (see \cite{joyalmoerdijk95, moerdijkpalmgren02, gambino08, awodeygambinolumsdainewarren09}). In particular, we should point out that \cite{awodeygambinolumsdainewarren09} establishes the more general result that they are stable under sheaves for a Lawvere-Tierney topology. Nevertheless, it is not quite true that our results are a special case of theirs, because, to achieve this generality, they work in a setting which has full (not just bounded) exactness. In addition, as we already mentioned in the introduction, it is not true that the fullness axiom {\bf (F)} is stable under sheaves for a Lawvere-Tierney topology.

\subsection{Sheafification}

Our next theorem shows the existence of a sheafification functor, a Cartesian left adjoint to the inclusion of sheaves in presheaves. The proof relies in an essential way on the assumption of bounded exactness and on the fact that our site is presentable. 

\begin{theo}{sheafiffunctor}
The inclusion
\diag{ i_*: \shct{\ct{E}}{\ct{C}} \ar@{>->}[r] & \pshct{\ct{E}}{\ct{C}} } 
has a Cartesian left adjoint $i^*$ (a ``sheafification functor'').
\end{theo}
\begin{proof}
We verify that it is possible to imitate the standard construction (see \cite[Section III.5]{maclanemoerdijk92}).

Let $P$ be a presheaf. A pair $(R, x)$ will be called a \emph{compatible family} on $a \in \ct{C}_0$, if $R$ is a covering sieve on $a$, and $x$ specifies for every $\func{f}{b}{a} \in R$ an element $x_f \in P(b)$, such that for any $\func{g}{c}{b}$ the equality $(x_f) \cdot g = x_{fg}$ holds. Because {\bf ($\Pi$E)} holds and sieves are small, by definition, there is an object of compatible families. Actually, the compatible families form a presheaf Comp($P$) with restriction given by
\[ (R, x) \cdot f = (f^*R, x \cdot f), \]
where $(x \cdot f)_g = x_{fg}$.

We define an equivalence relation on Comp($P$) by declaring two compatible families $(R,x)$ and $(T, y)$ on $a$ equivalent, when there is a covering sieve $S \subseteq R \cap T$ on $a$ with $x_f = y_f$ for all $f \in S$. Since the site is assumed to be presentable, this quantification over the (large) collection of covering sieves $S$ on $a$, can be replaced with a quantification over the small collection of basic covering sieves on $a$. Therefore the equivalence relation is bounded and has a quotient $P^+$. This object $P^+$ is easily seen to carry a presheaf structure in such a way that the quotient map ${\rm Comp}(P) \to P^+$ is a morphism of presheaves.

First claim: $P^+$ is separated. Proof: Suppose two elements $[R, x]$ and $[S, y]$ of $\psh{P}^+(a)$ agree on a cover $T$. Pick representatives $(R, x)$ and $(S, y)$, and define:
\[ Q = \{ \func{f}{b}{a} \in R \cap S \, : \, x_f = y_f \}. \]
Once we show that $Q$ is covering, we are done. But this follows immediately from the local character axiom for sites: for any $f \in T$, the sieve $f^*Q$ is covering, by assumption.

Second claim: when $P$ is separated, $P^+$ is a sheaf. Proof: Let $R$ be a covering sieve on $a$, and let compatible elements $p_f \in P^+(b)$ be given for every $\func{f}{b}{a} \in R$. Using the collection axiom, we find for every $f \in R$ a family of representatives $(R^{(f, i)}, x^{(f, i)})$ of $p_f$, with the variable $i$ running through some inhabited and \emph{small} index set $I_f$. Therefore
\[ S = \{ \,f \circ g \, : \, f \in R, i \in I_f, g \in R^{(f, i)} \, \} \]
is small; in fact, it is a covering sieve, by local character.

We now prove that for any two triples $(f \in R, i \in I_f, g \in R^{(f, i)})$ and $(f' \in R, i' \in I_{f'}, g' \in R^{(f', i')})$ with $fg = f'g'$, we must have $x^{(f, i)}_g = x^{(f', i')}_{g'}$. Since the elements $p_f$ are assumed to be compatible, the equality
\[ [R^{(f, i)}, x^{(f, i)}] \cdot g = p_{fg} = p_{f'g'} = [R^{(f', i')}, x^{(f', i')}] \cdot g' \]
holds. Hence the elements $x^{(f, i)}_g$ and $x^{(f', i')}_{g'}$ agree on a covering sieve. Since $P$ is assumed to be separated, this implies that the elements $x^{(f, i)}_g$ and $x^{(f', i')}_{g'}$ are in fact identical.

This argument shows that the definition $z_{fg} = x^{(f, i)}_g$ is unambiguous for $fg \in S$, and also that $(S, z)$ is a compatible family. As its equivalence class $[S, z]$ is the glueing of the family $p_f$ we started with, the second claim is proved.

From the construction it is clear that for any presheaf $P$ the sheaf $P^{++}$ has to be its sheafification. So we have shown that the construction of the sheafification functor carries through in the setting we are working in; that this assignment is moreover functorial as well as Cartesian is proved in the usual manner.
\end{proof}

\begin{theo}{sheavesHeytingcat}
\shct{\ct{E}}{\ct{C}} is a positive Heyting category.
\end{theo}
\begin{proof}
The category of sheaves has finite limits, because these are computed pointwise, as in presheaves. Using the following description of images and covers in categories of sheaves, one can easily show these categories have to be regular: the image of a map \func{F}{Y}{X} of sheaves consists of those $x \in X(a)$ that are ``locally'' hit by $F$, i.e., for which there is a sieve $S$ covering $a$ such that for any $\func{f}{b}{a} \in S$ there is an element $y \in Y(b)$ with $F(y) = x \cdot f$. Therefore a map \func{F}{Y}{X} is a cover, if for every $x \in X(a)$ there is a sieve $S$ covering $a$ and for any $f: b \to a \in S$ an element $y \in Y(b)$ such that $F(y) = x \cdot f$ (such maps are also called \emph{locally surjective}).

The Heyting structure in sheaves is the same as in presheaves, so the universal quantification of $A \subseteq Y$ along \func{F}{Y}{X} is given by the formula \refeq{forallinpresh}. Indeed, from this description it is readily seen that belonging to $\forall_F (A)$ is a local property.

The sums in sheaves are obtained by sheafifying the sums in presheaves. They are still disjoint and stable, because the sheafification functor is Cartesian.
\end{proof}

\subsection{Small maps in sheaves}

We will now define two classes of maps in the categories of sheaves, those which are pointwise small and those which are locally small. Using that {\bf (F)} holds in \ct{E} and the fact that the site is presentable, we will then show that they coincide. But before we define these two classes of maps, note that we have the following diagram of functors:
\diag{ \ct{E}/\ct{C}_0 \ar@<.5ex>[rr]^{\pi_!} \ar@<.5ex>[dr]^{\rho_!} & & \pshct{\ct{E}}{\ct{C}} \ar@<.5ex>[ll]^{\pi^*} \ar@<-.5ex>[dl]_{i^*} \\
& \shct{\ct{E}}{\ct{C}}, \ar@<.5ex>[lu]^{\rho^*} \ar@<-.5ex>[ur]_{i_*} & }
where the maps $\rho^*$ and $\rho_!$ are defined as the composites of $\pi$ and $i$ via the diagram. So $\rho^*$ is the forgetful functor, $\rho_!$ is defined as
\[ \rho_! X = i^* \pi_! X, \]
and they are adjoint. It follows immediately from the maximality axiom for sites that the components of the counit $\rho_! \rho^* \rTo 1$ are covers.

One final remark before we give the definitions. We have seen that any pair of maps $(r, s)$ in \ct{E} making
\diag{
B \ar[d]_r \ar[r]_s \ar@/^1pc/[rr]^{\sigma_B} & \ct{C}_1 \ar[d]^{\rm cod} \ar[r]_{\rm dom} & \ct{C}_0 \\
A \ar[r]_{\sigma_A} & \ct{C}_0.}
commute determines a map $(r, s)_!: \pi_! B \to \pi_! B$ of presheaves. Therefore it also determines a map $i^* (r, s)_!: \rho_! B \to \rho_! A$ of sheaves, but note that now not all maps $\rho_! B \to \rho_! A$ will be of this form, in contrast to what happened in the presheaf case.

Finally, the two classes of maps are defined as:
\begin{enumerate}
\item The pointwise definition: a morphism \func{F}{B}{A} of sheaves is \emph{pointwise small}, when $\rho^* F$ is a small map in $\ct{E}/\ct{C}_0$.
\item The local definition (as in \cite{joyalmoerdijk95}):  a morphism \func{F}{B}{A} of sheaves is \emph{locally small} in case it is covered by a map of the form $i^*(r, s)_!$ where $r$ is a small map in \ct{E}.
\end{enumerate}
That these two classes of maps coincide will follow from the next two propositions, both whose proofs use the fullness axiom.

\begin{prop}{sheafpressmallness}
The sheafification functor $i^*$ preserves pointwise smallness: if $F$ is a (pointwise) small map of presheaves, then $i^* F$ is a pointwise small map of sheaves. 
\end{prop}
\begin{proof}
To prove the proposition, it suffices to show that the $(-)^+$-construction preserves smallness. So let \func{F}{P}{Q} be a (pointwise) small morphism of presheaves and $q$ be an element of $Q^+(a)$, i.e. $q = [R, x]$ where $R$ is a sieve and $(x_f)_{f \in R}$ is a family of compatible elements. The fibre of $F^+$ over $q$ consists of equivalence classes of all those compatible families $(S, y)$ on $a$ such that $(S, F(y) )$ and $(R, q)$ are equivalent (by $F(y)$ we of course mean the family given by $F(y)_f = F(y_f)$). Because every such equivalence class is represented by a compatible family $(S, y)$ where $S$ is a \emph{basic} covering sieve contained in $R$ and $F(y_f) = x_f$ for all $f \in S$, the fibre of $F$ over $q$ is covered by the object:
\[ \sum_{S \in {\rm BCov}(a), S \subseteq R} \, \prod_{f \in S} F^{-1}(x_f). \]
It follows from the fullness axiom in \ct{E} that this object is small (actually, the exponentiation axiom {\bf ($\Pi$S)} would suffice for this purpose) and then it follows from the quotient axiom {\bf (A6)} that the fibre of $F$ over $q$ is small as well.
\end{proof}

\begin{prop}{pointwisesmclosedundercovering}
The pointwise small maps in sheaves are closed under covered maps: if
\diag{ X \ar[d]_{F} \ar[r] & A \ar[d]^{G} \\
Y \ar@{->>}[r]_P & B}
is a covering square of sheaves (i.e., $P$ and the induced map $X \to Y \times_B A$ are locally surjective) and $F$ is pointwise small, then also $G$ is pointwise small.
\end{prop}
\begin{proof}
To make the proof more perspicuous, we will split the argument in two: first we show closure of pointwise small maps under quotients and then under descent. 

So suppose first that we have a commuting triangle of sheaves
\diag{Y \ar@{->>}[rr]^G  \ar[dr] _F& & X \ar[dl]^H \\
& B, & }
with $F$ pointwise small and $G$ locally surjective. Fix an element $b \in B(c)$. The fullness axiom in \ct{E} implies that for any basic covering sieve $S \in {\rm BCov}(c)$ there is a small generic family $P^S_b$ of \emph{mvs}s of the obvious (small) projection map
\diag{ p^S_b: Y^S_b = \{ (f: d \to c \in S, y \in Y(d)) \, : \, F_d(y) = b \cdot f \} \ar[r] & S,}
such that any \emph{mvs} of this map is refined by one in $P^S_b$ (recall that an \emph{mvs} of $p^S_b$ would be a subobject $L \subseteq Y^S_b$ such that the composite $L \subseteq Y^S_b \to S$ is a small cover).  Strictly speaking, the fullness axiom says that for every $S \in {\rm BCov}(c)$ such a generic \emph{mvs} exists,  not necessarily as a function of $S$. This does follow, however, using the collection axiom: for this axiom tells us that there is a small family $\{ P_i \, : \, i \in I^S_b \}$ of such \emph{mvs}s for every $S$. So we can set $P^S_b = \bigcup_{i \in I^S_b} P_i$ to get a generic \emph{mvs} of $p^S_b$ as a function of $S$.

Call an element $L \in P^S_b$ \emph{compatible after $G$}, if for any pair of elements $(f: d \to c, y)$ and $(f': d' \to c, y')$ in $L$ we have
\[ \forall g: e \to d, g': e \to d' \, ( \, fg = f' g' \Rightarrow G_d(y) \cdot g = G_{d'}(y') \cdot g' \, ). \]
Note that there is a map
\[ q: \sum_{S \in {\rm BCov}(c)} \{ L \in P^S_b \, : \, L \mbox{ compatible after } G  \, \}  \to H_c^{-1}(b),  \]
which one obtains by sending $(S, L)$ to the glueing of the elements $\{ G_d(y) \, : \, (f: d \to c, y) \in L \}$ in $X$. The domain of this map $q$ is small, so the desired result will follow, once we show that this map is a cover. For this we use the local surjectivity of $G$.

Local surjectivity of $G$ means that for every $x \in X(c)$ in the fibre over $b \in B(c)$, there is a basic covering sieve $S \in {\rm BCov}(c)$ such that
\[ \forall f: d \to c \in S \, \exists y \in Y(d) \, : \, G_d(y) = x \cdot f.  \]
But $G_d(y) = x \cdot f$ implies that $F_d(y) = b \cdot f$, so
\[ \{ (f: d \to c, y \in Y(d)) \, : \, G_d(y) = x \cdot f \} \]
is an \emph{mvs} of $p^S_b$ and therefore it is refined by an element of $P^S_b$. Since this element must be compatible after $G$, we have shown that $q$ is a cover. 

Second, suppose we have a pullback square of sheaves
\diag{ X \ar[d]_{F} \ar[r]^Q & A \ar[d]^{G} \\
Y \ar@{->>}[r]_P & B,}
where $F$ is pointwise small and $P$ and $Q$ are locally surjective. Again, for any $b \in B(c)$ and basic covering sieve $S$ of $c$, let $p^S_b$ be the map
\diag{ p^S_b: Y^S_b = \{ (f: d \to c \in S, y \in Y(d)) \, : \, P_d(y) = b \cdot f \} \ar[r] & S,}
as above. Furthermore, let ${\rm mvs}(p^S_b)$ be the object of \emph{mvs}s of $p^S_b$ and set
\begin{eqnarray*}
Y'(c) & = & \sum_{b \in B(c)} \sum_{S \in {\rm BCov}(c)} {\rm mvs}(p^S_b), \\
X'(c) & = & \sum_{(b, S, L) \in Y'(c)} \{ \, k \in \prod_{(f: d \to c, y) \in L} F_d^{-1}(y) \, : \, k \mbox{ compatible after } Q \, \},
\end{eqnarray*}
where we call $k \in \prod_{(f: d \to c, y) \in L} F_d^{-1}(y)$ \emph{compatible after $Q$}, if for any $(f: d \to c, y)$ and $ (f': d' \to c, y')$ in $L$ we have
\[ \forall g: e \to d, g': e \to d'  \in \ct{C} \, ( fg = f' g' \Rightarrow Q_d(k_{(f, y)}) \cdot g = Q_{d'}(k_{(f', y')}) \cdot g'  ). \]
This leads to a commuting square in $\ct{E}/\ct{C}_0$
\diag{ X'(c) \ar@{->>}[r]^{Q'_c} \ar[d]_{F'_c} & A(c) \ar[d]^{G_c} \\
Y'(c) \ar@{->>}[r]_{P'_c} & B(c),}
in which $P'$ and $F'$ are the obvious projections and $Q'$ sends $(b, S, L, k)$ to the glueing of $\{ Q_d(k_{(f, y)}) \, : \, (f, y) \in L \}$. The square is a pullback in which the map $P'$ is a cover (this uses the collection axiom) and $F'$ is small, so that $G_c$ is a small map by descent {\bf (A2)}  in $\ct{E}/\ct{C}_0$. This completes the proof. 
\end{proof}

\begin{theo}{threecoincidesheaves}
The pointwise small maps and locally small maps of sheaves coincide.
\end{theo}
\begin{proof}
That all locally small maps of sheaves are also pointwise small follows from the previous two propositions. To prove that all pointwise small maps are also locally small we use that the pointwise and locally small maps coincide in presheaves. 

So consider a pointwise small map \func{F}{B}{A} of sheaves. Since $i_* F$ is a pointwise small map of presheaves, there is a small map of presheaves $(k, l)_!$ with $k$ small in \ct{E} such that
\diag{\pi_! X \ar[d]_{(k, l)_!} \ar[r] & i_* B \ar[d]^{i_* F} \\
\pi_! Y \ar[r] & i_* A } is a covering square in presheaves. Applying sheafification $i^*$ and using that $i^* i_* \cong 1$, we obtain a diagram of the desired form.
\end{proof}

\begin{coro}{pointwisecoveredbylocallysmallsheaves}
Any pointwise small map is covered by one of the form $i^*(r, s)_!$ with $r$ small in \ct{E}. In fact, every composable pair $(G, F)$ of pointwise small maps of sheaves fits into a double covering square of the form
\diag{ \rho_! C \ar[d]_{i^*(k, l)_!} \ar@{->>}[r] & Z \ar[d]^G \\
\rho_! B \ar[d]_{i^*(r, s)_!}  \ar@{->>}[r] & Y \ar[d]^F \\
\rho_! A \ar@{->>}[r] & X, }
in which $k$ and $r$ are small in \ct{E}.
\end{coro}
\begin{proof}
Immediate from the previous theorem and the corresponding fact for presheaves (\refcoro{pointwisecoveredbylocallysmall}).
\end{proof}

Henceforth we can therefore use the term ``small map'' without danger of ambiguity. The first thing to do now is to show that the small maps in sheaves really satisfy the axioms for a class of small maps.

\begin{theo}{basicpropofsh}
The small maps in sheaves satisfy axioms {\bf (A1-9)}.
\end{theo}
\begin{proof}
Again, we postpone the proof of the collection axiom {\bf (A7)} (it will be \reftheo{collstablesh}). Because limits in sheaves are computed as in presheaves, {\bf (A1)} and {\bf (A9)} are inherited from presheaves. Colimits in sheaves are computed by sheafifying the result in presheaves, hence the axioms {\bf (A3)} and {\bf (A4)} follow from \refprop{sheafpressmallness}. That pointwise small maps are closed under covered maps was \refprop{pointwisesmclosedundercovering}: this disposes of {\bf (A2)} and {\bf (A6)}. Pointwise small maps are closed under composition, so {\bf (A5)} holds as well. Finally, since universal quantification in sheaves is computed as in presheaves, the axiom {\bf (A8)} holds in sheaves, because it holds in presheaves.
\end{proof}

\begin{theo}{manythingsstableinsh}
The following axioms are inherited by sheaf models: bounded exactness, representability, {\bf (NE), (NS), ($\Pi$E), ($\Pi$S)}, {\bf (M)} and {\bf (PS)}.
\end{theo}
\begin{proof}
Bounded exactness is inherited by sheaf models, since one can sheafify the quotient in presheaves. Representability is inherited for the same reason: one sheafifies the universal small maps in presheaves. Also the natural numbers object in sheaves is obtained by sheafifying the natural numbers object in presheaves, so {\bf (NE)} and {\bf (NS)} are inherited by sheaf models. Since $\Pi$-types in presheaves are computed as in sheaves and {\bf ($\Pi$E)} and {\bf ($\Pi$S)} are inherited by presheaf models, they will also be inherited by sheaf models. Finally, since monos in sheaves are pointwise, {\bf (M)} is inherited as well. 

The \spower-functor in sheaves is obtained by quotienting the \spower-functor in presheaves (see \refprop{someaxiomsstunderpresh}) by the following equivalence relation (basically, bisimulation understood as in sheaves): if $A, A' \subseteq {\bf y}c \times X$, then $A \sim A'$ if for all $ (f: b \to c, x) \in A(c)$, the sieve 
\[ \{ \, g: a \to b \, : \, (fg, x \cdot g) \in A' \, \} \]
covers $b$, and for all $(f': b' \to c, x') \in A'(c)$ the sieve 
\[ \{ \, g': a' \to b' \, : \, (f'g', x' \cdot g') \in A \, \} \]
covers $b'$.

One easily verifies that this defines an equivalence relation in presheaves; moreover, it is bounded, since the site is assumed to be presentable. Its quotient $P$ has the structure of a sheaf (as we have seen several times, to construct the glueing one uses the collection axiom to select small collections of representatives from each equivalence class). One defines the relation $\in_X \subseteq X \times P$ on an object $c \in \ct{C}$ by putting for any $x \in X(c)$ and $A \in \spower(X)(c)$,
\[ x \in [A] \Longleftrightarrow \mbox{the sieve } \{ f: d \to c \, : \,  (f, x \cdot f) \in A \} \mbox{ covers } c.\]
A straightforward verification establishes that this is indeed the power class object of $X$ in sheaves. Hence the axiom {\bf (PS)} is inherited by sheaf models.
\end{proof}

In the coming two subsections we will discuss the collection and fullness axioms and W-types in sheaf categories.

\subsection{Collection and fullness in sheaves}

\begin{theo}{collstablesh}
The collection axiom {\bf (A7)} is inherited by sheaf models.
\end{theo}
\begin{proof}
Let \func{F}{M}{N} be small map and $E: Y \rTo M$ be a cover in sheaves (i.e. $E$ is locally surjective). Without loss of generality we may assume that $F$ is of the form $i^* (k, l)_!: \rho_! B \to \rho_! A$. 

If the map $Q: X \to \pi_! B$ of presheaves is obtained by pulling back the map $i_*E$ along the component of the unit $1 \to i_*i^*$ at $\pi_!B$ as in
\diag{ X \ar[d]_Q \ar[r] & i_* Y \ar[d]^{i_* E} \\
\pi_! B \ar[r] & i_*\rho_! B = i_*i^* \pi_! B,}
then this map $Q$ also has to be locally surjective. This means that for the following object in \ct{E}
\begin{eqnarray*}
C & = & \{ \, (b \in B, S \in {\rm Cov}(c)) \, : \,  \sigma_B(b) = c \mbox{ and } \\
& & (\forall f: d \to c  \in S) \, (\exists x \in X(d)) \, (Q(x) = (b, f))  \},
\end{eqnarray*}
the obvious projection $s_0: C \to B$ is a cover. Therefore we can apply the collection axiom in \ct{E} to obtain a covering square of the form:
\diaglab{covsq1}{ V \ar[r]^{s_1} \ar[d]_l & C \ar@{->>}[r]^{s_0} & B \ar[d]^{k} \\
U \ar[rr]_{r_0} & & A,}
with $l$ small in \ct{E}. We wish to apply the collection axiom again. For this purpose, define the following two objects in \ct{E}:
\begin{eqnarray*}
V'&  = & \{ \, (v \in V, f \in \ct{C}) \, : \, \mbox{if } s_1(v) = (b, S),\mbox{then } f \in S \, \},  \\
W & = & \{ \, (v \in V, f: d \to c, x \in X(d)) \, : \, \mbox{if } s_1(d) = (b, S), \\
& & \mbox{ then } f \in S \mbox{ and } Q_d(x) = (b, f) \, \},
\end{eqnarray*}
and let $s_3: W \to V'$ and $s_2: V' \to V$ be the obvious projections. $s_3$ is a cover (essentially by definition of $C$), and the composite $l' = ls_2$ is small. So we can apply collection to obtain a covering square in \ct{E} 
\diaglab{covsq2}{ J \ar[r]^{s_4}  \ar[d]_{m} & W \ar@{->>}[r]^{s_3} & V' \ar[d]^{l'} \\
I \ar[rr]_{r_1} & & U, }
in which $m$ is small. Writing $r= r_0r_1$ and $s = s_0s_1s_2s_3s_4$, we obtain a commuting square
\diag{J \ar[r]^s \ar[d]_{m} & B \ar[d]^{k} \\
I \ar[r]_r & A,}
with every $j \in J$ determining an element $b \in B$, a sieve $S$ on $\sigma_B(b)$, an arrow $f \in S$ and an element $x \in X (\mbox{dom } f)$ such that $Q(x) = (b, f) \in \pi_! B$. If for such an element $j \in J$ we put $\rho_J(j) = {\rm dom}(f), t_j = f, n_j = l_b \circ f$ and for every $i \in I$ we define $\sigma_I(i) = \sigma_A(ri)$, then we obtain a square of presheaves:
\diag{ \pi_! J \ar[r]^{(s, t)_!} \ar[d]_{(m, n)_!} & \pi_! B \ar[d]^{(k, l)_!} \\
\pi_! I \ar[r]_{\pi_!r} & \pi_! A.}
To see that it commutes, we chase an element around the two sides of the diagram and it suffices to do that for an element of the form $(j, \id)$. So \[ \pi_! r (m, n)_! (j, \id) = \pi_! r (m j, l_b \circ f) = (rm j, l_b \circ f), \]
and $(k, l)_! (s, t)_! (j, \id) = (k, l)_! (sj, f) = (ksj, l_b \circ f)$.

We claim that sheafifying the square gives a covering square. Since $r$ is a cover and $\rho_!$ preserves these, this means that we have to show that the map from $\pi_! J$ to the pullback of the above square is locally surjective. \reflemm{propofshriek}.4 tells us that we may assume that the pullback is of the form  $\pi_!(I \times_A B)$ with $\sigma_{I \times_A B}(i, b) = \sigma_B(b)$. The induced map \func{K}{\pi_! J}{\pi_!(I \times_A B)} sends $(j, g)$ to $((mj, sj), f \circ g)$, where $f$ is the element in $\ct{C}_1$ determined by $j \in J$ as above. To show that this map is locally surjective, it suffices to prove that every element $((i, b), \id) \in \pi_!(I \times_A B)$ is locally hit by $K$. The element $i \in I$ determines an element $r_1i \in U$, and since \refdiag{covsq1} is a covering square, we find a $v \in V$ with $lv = r_1i$ and $s_0s_1v = b$, hence a covering sieve $S$ on $\rho_B(b)$. Moreover, since \refdiag{covsq2} is a covering square, we find for every $f \in S$ an element $j \in J$ such that $m(j) = i$ and $s(j) = b$. Then $K(j, \id) = ((i, b), f) = ((i, b), \id) \cdot f$, which proves that $K$ is locally surjective.

To complete the proof, we need to show that \func{(s, t)_!}{\pi_!J}{\pi_! B} factors through \func{Q}{X}{\pi_! B}. There is a map $(p, q): J \rTo \pi^* X$ which sends every $j \in J$ to the $x \in X(\mbox{dom } f)$ that it determines. Its transpose $(p, q)_!$ sends $(j, \id)$ to $x \in X$ which in turn is sent by $Q$ to $Q(x) = (sj, f) = (s, t)_! (j, \id)$. Therefore $(s, t)_! = Q(p, q)_!$.
\end{proof}

\begin{theo}{Fstableundersheaves}
(Assuming \ct{C} has chosen pullbacks.) The fullness axiom {\bf (F)} is inherited by sheaf models.
\end{theo}
\begin{proof}
In view of \reflemm{redlemmforfull} and \refcoro{pointwisecoveredbylocallysmallsheaves}, it will suffice to show that there exists a generic \emph{mvs} for any map of the form \func{i^*(k, \kappa)_!}{\rho_! B}{\rho_! A}, living over some object of the form $\rho_! X$ via some map \func{i^*(r, \rho)_!}{\rho_! A} {\rho_! X}, with $k$ and $r$ small.

We first construct the generic \emph{mvs} $P$. To this end, define:
\begin{eqnarray*}
S_0 & = & \{ (a \in A, \alpha :d \to c, S \in {\rm BCov}(\alpha^*\sigma_A(a)) )\, : \, \sigma_X(ra) = c \} \\
M_0 & = & \{ (a \in A, \alpha:d \to c, S \in {\rm BCov}(\alpha^*\sigma_A(a)), \beta \in S) \, : \, \sigma_X(ra)  = c \} \\
B_0 & = & \{ (b \in B, \alpha: d \to c, S \in {\rm BCov}(\alpha^*\sigma_A(kb)), \beta \in S, \gamma \in \ct{C}_1 \, :  \\
& & \sigma_X(rkb) = c, \alpha^*\kappa_b \circ \gamma = \beta \}
\end{eqnarray*}
(In the definition of $S_0$ and $M_0$ we have used that any pair consisting of a map $\alpha: d \to c \in \ct{C}$ and element $a \in A$ with $\sigma_X(ra) = c$ determines a pullback diagram
\diag{ \alpha^*\sigma_A(a) \ar[r] \ar[d]_{\alpha^*\rho_{a}} & \sigma_A(a) \ar[d]^{\rho_{a}} \\
d \ar[r]_\alpha & c }
in \ct{C}; in the definition of $B_0$ we have used that any pair consisting of a map $\alpha: d \to c \in \ct{C}$ and element $b \in B$ with $\sigma_X(rkb) = c$ determines a double pullback diagram
\diag{ \alpha^*\sigma_B(b) \ar[r] \ar[d]_{\alpha^*\kappa_b} & \sigma_B(b) \ar[d]^{\kappa_b} \\
\alpha^*\sigma_A(kb) \ar[r] \ar[d]_{\alpha^*\rho_{kb}} & \sigma_A(kb) \ar[d]^{\rho_{kb}} \\
d \ar[r]_\alpha & c }
in \ct{C}.) One easily checks that all the projections in the chain
\diag{B_0 \ar[r] & M_0 \ar[r] & S_0 \ar[r] & A \ar[r] & X }are small.

For the construction of $P$, we first build a generic \emph{mvs} for $S_0 \rTo A$ over $X$. This means we have a cover \func{n}{W}{X} and a small map \func{m_0}{Z_1}{W}, together with a generic \emph{mvs} $P_1$ for $S_0 \rTo A$ over $Z_1$, as in the diagram
\diag{ P_1 \ar@{>->}[r] \ar@/_/@{->>}[dr] & S_1 \ar[d] \ar[rr] & & S_0 \ar[d] \\
& A_1 \ar[d] \ar[rr] & & A \ar[d] \\
& Z_1 \ar[r]_{m_0} & W \ar@{->>}[r]_n  & X, }
where the rectangles are understood to be pullbacks. Next, we pull $B_0 \to M_0 \to S_0$ back along $P_1 \to S_0$ and obtain the diagram
\diag{ B_1 \ar[r] \ar[d] & B_0 \ar[d] \\
M_1 \ar[d] \ar[r] & M_0 \ar[d] \\
P_1 \ar[r] & S_0. }
Then we build a generic \emph{mvs} for $B_1 \to M_1$ over $Z_1$. This we obtain over an object $Z_2$ via a small map $Z_2 \rTo W'$ and a cover $W' \rTo Z_1$. Without loss of generality, we may assume that the latter map $W' \to Z_1$ is the identity. (Proof: apply the collection axiom to the small map $Z_1 \rTo W$ and the cover $W' \rTo Z_1$ to obtain a small map $S \rTo R$ covering the morphism $Z_1 \rTo W$. \reflemm{diffgenmvss} tells us that there lives a generic \emph{mvs} for $S_0 \rTo A$ over $S$ as well. By another application of \reflemm{diffgenmvss}, there lives a generic \emph{mvs} for $B_1 \to M_1$ over $T$, if $T \to S$ is the pullback of $Z  \rTo W'$ along the map $S \rTo W'$.) So we may assume there is a small map $m_1: Z_2 \rTo Z_1$, such that over $Z_2$ there is a generic \emph{mvs} $P_2$ for $B_1 \to M_1$, as in the following diagram
\diag{ P_2 \ar@{>->}[r] \ar@/_/@{->>}[dr] & B_2 \ar[d] \ar[r] & B_1 \ar[rr] \ar[d] & & B_0 \ar[d] \\
& M_2 \ar[d] \ar[r] & M_1 \ar[rr] \ar[d] & & M_0 \ar[d] \\
& Z_2 \ar[r]^{m_1}  \ar@/_1pc/[rrr]_t & Z_1 \ar[r]^{m_0}  & W \ar@{->>}[r]^n & X, }
where all the rectangles are supposed to be pullbacks. Fo convenience, write $t = nm_0m_1$.

We make some definitions. First of all, let
\begin{eqnarray*}
Z & = & \{ (z_2 \in Z_2, \delta: \, d \to c) \, : \, \sigma_X(t(z_2)) = c \mbox{ and } \\
& & (\forall a \in A_{t(z_2)}) \, (\exists S \in {\rm BCov}(\delta^* \sigma_A(a)) \, (m_1(z), a, \delta, S) \in P_1 \}.
\end{eqnarray*}
Furthermore, we write $m_2: Z \to Z_2$ for the obvious projection and put $m = m_0m_1m_2$. Finally, we let $P_3$ be the pullback of $P_2$ along $m_2$.

We wish to construct a diagram of presheaves of the form:
\diag{ P \ar@{>->}[r] \ar@/_/[dr]
 & \pi_!(Z \times_X B) \ar[d] \ar[rr] & & \pi_! B \ar[d]^{(k, \kappa)_!} \\
& \pi_! (Z \times_X A) \ar[d] \ar[rr] & & \pi_! A \ar[d]^{(r, \rho)_!}  \\
& \pi_! Z \ar[r]_{(m, \mu)_!}  & \pi_! W \ar[r]_{\pi_!n} & \pi_! X, & }
which we can do by putting $\sigma_Z(z_2, \delta) = {\rm cod}(\delta)$ and $\mu_{(z_2, \delta)} = \delta$. Note that $\pi_!n$ is a cover and $(m, \mu)_!$ is small. In addition, $P$ is defined by
saying that an element  $(z \in Z, b \in B, \eta: c \to d) \in \pi_!(Z \times_X B)(c)$ belongs to $P(c)$ if
\begin{quote}
there is a sieve $S \in {\rm BCov}(\mu_z^* \sigma_A(kb))$, a map $\beta \in S$ and a map $\gamma \in\ct{C}_1$ such that $(z, b, \mu_z, S, \beta, \gamma)$ belongs to $P_3$ and $\eta$ factors through $\gamma$.
\end{quote}
By construction, the map $P \rTo \pi_! (Z \times_X A)$ is locally surjective. By sheafifying the whole diagram, we therefore obtain an \emph{mvs} $i^* P$ for $i^*(k, \kappa)_!$ over $\rho_! Z$ in the category of sheaves. The remainder of the proof will show it is generic.

To that purpose, let $V \rTo \rho_! W$ be a map of sheaves and $Q$ be an \emph{mvs} for $i^* (k, \kappa)_!$ over $V$. Let $\overline{Y}$ be the pullback in presheaves of $V$ along the map $\pi_! W \to \rho_! W$ and cover $\overline{Y}$ using the counit $\pi_! \pi^* \overline{Y} \to \overline{Y}$. Writing $Y = \pi^* \overline{Y}$, this means we have a commuting square of presheaves
\diag{ \pi_!Y \ar[r]^{(l, \lambda)_!} \ar[d] & \pi_! W \ar[d] \\
V \ar[r]  & \rho_! W, }
in which the vertical arrows are locally surjective and the top arrow is of the form $(l, \lambda)_!$. Finally, let $\overline{Q}$ be the pullback of $Q$ along $\pi_! Y \to V$. This means we have the following diagram of presheaves:
\diag{ \overline{Q} \ar@{>->}[r] \ar@/_/[dr]
 & \pi_!(Y \times_X B) \ar[d] \ar[rr] & & \pi_! B \ar[d]^{(k, \kappa)_!} \\
& \pi_!(Y \times_X A) \ar[d] \ar[rr] & & \pi_! A \ar[d]^{(r, \rho)_!}  \\
& \pi_! Y \ar[r] _{(l, \lambda)_!} & \pi_! W \ar[r]_{\pi_!n} & \pi_! X, &, }
where the rectangles are pullbacks, computed, as usual, using \reflemm{propofshriek}.5 (so $\sigma_{Y \times_X A}(y, a) = \lambda_y^* \sigma_A(a)$ and $\sigma_{Y \times_X B}(y, b) = \lambda_y^* \sigma_B(b)$). The map $\overline{Q} \to \pi_! (Y \times_X A)$ is locally surjective, and therefore 
\begin{eqnarray*}
Q_1 & = & \{ \, (y, a, \lambda_y, S \in {\rm BCov}(\lambda_y^* \sigma_A(a))) \, : \\
& & \, (\forall \beta \in S) \, (\exists b \in B_a) \, (\exists \gamma \in \ct{C}_1) \, (y, b, \gamma) \in \overline{Q} \mbox{ and } \lambda_y^* \kappa_b \circ \gamma = \beta  \, \} \\
& = & \{ \, (y, a, \lambda_y, S \in {\rm BCov}(\lambda_y^* \sigma_A(a))) \, : \\
& & \, (\forall \beta \in S) \, (\exists b \in B_a) \, (\exists \gamma \in \ct{C}_1) \, (y, b, \gamma) \in \overline{Q} \mbox{ and } (b, \lambda_y, S, \beta, \gamma) \in B_0\, \}
\end{eqnarray*}
is an \emph{mvs} of $S_0 \to A$ over $Y$. By the genericity of $P_1$ this implies the existence of a map $v_1: U_1 \to Z_1$ and a cover $w_1: U_1 \to Y$ such that $m_0v_1 = l w_1$ and $v_1^* P_1 \leq w_1^* Q_1$ as \emph{mvs}s of $S_0 \to A_0$ over $U_1$. Note that this means that 
\begin{equation} \label{lifesaver}
 (v_1(u_1), a, \alpha, S) \in P_1 \Longrightarrow \alpha = \lambda_{w_1(u_1)}.
\end{equation}
Next, define the subobject $Q_2 \subseteq v_1^* B_1$ by saying for any element $(u_1 \in U_1, b \in B, S \in {\rm BCov}(\lambda_{w_1(u_1)}^*\sigma_A(kb)), \beta \in S, \gamma \in \ct{C}_1) \in v_1^* B_1$:
\[  (u_1, b, S, \beta, \gamma) \in Q_2 \Longleftrightarrow (w_1 u_1, b, \gamma) \in \overline{Q}({\rm dom}(\gamma)). \]
It follows from (\ref{lifesaver}) and the definition of $Q_1$ that $Q_2$ is a small \emph{mvs} of $B_1 \to M_1$ over $U_1$. Therefore there is a map \func{v_2}{U}{Z_2} and a cover \func{w_2}{U}{U_1} such that $v_1 w_2 = m_1 v_2$ and $v_2^* P_2 \leq w_2^* Q_2$. Note that (\ref{lifesaver}) implies that $v_2$ factors through $m_2: Z \to Z_2$ via a map $v: U \to Z$ given by $v(u) = (v_2(u), \lambda_{w_1w_2(u)})$.

If we put $w = w_1 w_2$, then $lw = lw_1w_2 = m_0v_1w_2= m_0m_1v_2 = m_0m_1m_2v = mv$. Since for each $u \in U$, $\sigma_Z(vu) = {\rm dom}(\lambda_{wu}) = \sigma_Y(wu)$, we may put $\sigma_U(u) = \sigma_Z(vu) = \sigma_Y(wu)$ and then $\pi_!w$ and $\pi_!v$ define maps $\pi_! U \to \pi_! Y$ and $\pi_! U \to \pi_! Z$, respectively, such that $(l, \lambda)_! \pi_! w = (m, \mu)_! \pi_!v$. Because $\pi_!w$ is a cover, the proof will be finished, once we show that $(\pi_!v)^* P \leq (\pi_! w)^* \overline{Q}$. 

To show this, consider an element $(u \in U, b \in B, \eta: c \to d \in \ct{C}) \in \pi_!(U \times_X B)(c)$ for which we have $(u, b, \eta) \in (\pi_! v)^* P(c)$. This means that $(vu, b, \eta) \in P(c)$ and hence that there is a sieve $S \in {\rm BCov}(\mu_{vu}^* \sigma_A(kb))$, a map $\beta \in S$ and a map $\gamma: e \to d \in \ct{C}_1$ such that $(vu, b, \mu_{vu}, S, \beta, \gamma) \in P_3$ and $\eta$ factors through $\gamma$. The former means that $(v_2u, b, \mu_{vu}, S, \beta, \gamma) \in P_2$ and since $v_2^* P_2 \leq w_2^* Q_2$, it follows that $(w_2 u, b, S, \beta, \gamma) \in Q_2$. By definition this means that $(wu, b, \gamma) \in \overline{Q}(e)$. Since $\overline{Q}$ is a presheaf, also $(w u, b, \eta) \in \overline{Q}(c)$ and hence $(u, b, \eta) \in (\pi_!w)^* \overline{Q}(c)$. This completes the proof.
\end{proof}

\begin{rema}{fullnesshinternallogic}
Again, one can also prove this result using the internal logic of categories of sheaves. Also to illustrate its power, we give one such proof here.

\begin{theo}{stabFforshinternallogic}
(Assuming \ct{C} has chosen finite products.) The fullness axiom {\bf (F)} is inherited by sheaf models.
\end{theo}
\begin{proof}
In view of \reflemm{redlemmforfull} and \refcoro{pointwisecoveredbylocallysmall}, we only need to build generic \emph{mvs}s for maps of the form $i^*(k, l)_!: \rho_! B \to \rho_! A$ in which $k$ is small, where $\rho_! A$ lies over some object of the form $\rho_! X$ via a map of the form $i^*(r, s)_!$ in which $r$ is small. Again, by replacing \ct{C} by $\ct{C}/\rho_!(X)$, we may assume that $X = 1 = \{ * \}$ and $\sigma_X(*) = 1$.

Note that for a fixed $c \in \ct{C}_0$ an \emph{mvs} of $i^*(k, l)_!$ over $i^*\ct{C}(-, c)$ as in
\diag{ P \ar@{ >->}[r] \ar@{->>}[dr] & \rho_!(B) \times i^*\ct{C}(-, c) \ar[d]^{i^*((k, l)_! \times \id)} \\
& \rho_!(A) \times i^*\ct{C}(-, c) }
satisfies
\begin{displaymath}
\begin{array}{c}
(\forall a \in A) \, (\exists S \in {\rm BCov}(\sigma_A(a) \times c)) \\ (\forall g: d \to \sigma_A(a) \times c \in S) \, (\exists b \in B) \, (\exists f: d \to \sigma_B(b)) \, (\exists h: d \to c)  \\ (k(b), l_b \circ f) = (a, \pi_1 \circ g), \pi_2 \circ g = h \mbox{ and } (b, f, h) \in P(d),
\end{array}
\end{displaymath}
or
\begin{equation} \label{mvsschoven}
\begin{array}{c}
(\forall a \in A) \, (\exists S \in {\rm BCov}(\sigma_A(a)) \times c) \\ (\forall g: d \to \sigma_A(a) \times c \in S) \, (\exists b \in B) \, (\exists f: d \to \sigma_B(b)) \\ k(b) =a, l_b \circ f = \pi_1 \circ g \mbox{ and } (b, f, \pi_2 \circ g) \in P(d).
\end{array}
\end{equation}
We first apply fullness in \ct{E} to the map
\[ \{ (a \in A, c \in \ct{C}_0, S \in {\rm BCov}(\sigma_A(a) \times c)) \} \to A: (a, c, S) \to a \]
to obtain a generic small family of \emph{mvs}s $(Q_j)_{j \in J}$.

Writing for every $j \in J$
\begin{eqnarray*}
\bar{Q}_j & := & \{ (a, c, S, g) \, : \, (a, c, S) \in Q_j, g \in S \}, \\
\tilde{Q}_j & := & \{ (a, c, S, g, b, f) \, : \, (a, c, S) \in Q_j, g \in S, b \in B, k(b) = a, l_bf = \pi_1g \},
\end{eqnarray*}
we have an obvious projection
\[ \tilde{Q}_j \to \bar{Q}_j. \]
Applying fullness and using the collection axiom we obtain generic families $\{ P_{ij} \}_{i \in I_j}$ for these maps as well. (The collection axiom is employed here to obtain these generic families as a function of $j$.)

For fixed $c \in \ct{C}, j \in J$ and $i \in I_j$ the object $P_{ij}$ determines a subsheaf
\[ \hat{P}_{ij}^{(c)} \subseteq \rho_!(B) \times i^*\ct{C}(-, c) \]
generated by those elements $(b, f, \pi_2 \circ g)$ for which there is a basic covering sieve $S$ such that $(k(b), c, S, g, b, f) \in P_{ij}$. Again, we only take those which are \emph{mvs}s, i.e., map in a locally surjective manner to $\rho_!(A) \times i^*\ct{C}(-, c)$.

Now suppose $c_0$ is arbitrary and $R \subseteq \rho_!(B) \times \ct{C}(-, c_0)$ is an \emph{mvs}. This means that (\ref{mvsschoven}) holds with $c = c_0$. Hence there is a $Q_j$ with
\begin{equation} \label{niceinclusion}
\begin{array}{rcl}
Q_j & \subseteq & \{ \, (a, c_0, S) \, : \, (\forall g: d \to \sigma_A(a) \times c \in S) \, (\exists b \in B) \, (\forall f: d \to \sigma_B(b)) \\
& & k(b) = a, l_b \circ f = \pi_1 \circ g \mbox{ and } (b, f, \pi_2 \circ g) \in R(d) \, \}
\end{array}
\end{equation}
and an $i \in I_j$ with
\begin{eqnarray*}
P_{ij} & \subseteq & \{ \, (a, c_0, S, g, b, f) \, : \, (b, f, \pi_2 \circ g) \in R(d) \, \}. 
\end{eqnarray*}
It is clear that $\hat{P}_{ij}^{(c_0)} \subseteq R$, so it remains to verify that $\hat{P}_{ij}^{(c_0)}$ is an \emph{mvs}.

We check (\ref{mvsschoven}) for $c = c_0$. So take $a \in A$. We want to show that the generator $(a, \pi_1, \pi_2) \in \rho_!(A) \times i^*\ct{C}(-, c_0)$ is locally hit by $i^*((k, l)_! \times \id)$. Because $Q_j$ is an \emph{mvs}, there are $e \in \ct{C}_0$ and $S \in {\rm BCov}(\sigma_A(a) \times e)$ such that $(a, e, S) \in Q_j$; morever, we must have $e = c_0$, because (\ref{niceinclusion}) holds. Since $P_{ij}$ is an \emph{mvs} we know that for every $g \in S$ there are $b \in B$ and $f \in \ct{C}_1$ with $(a, c_0, S, g, b, f) \in P_{ij}$. In particular $(a, c_0, S, g, b, f) \in \tilde{Q}_{j}$, so $k(b) = a$ and $l_b \circ f = \pi_1 \circ g$. By construction $(b, f, \pi_2 \circ g) \in \hat{P}_{ij}^{(c_0)}$ and for this element the equation 
\[ i^*((k, l)_! \times \id)(b, f, \pi_2 \circ g) = (k(b), l_b \circ f, \pi_2 \circ g) = (a, \pi_1, \pi_2) \cdot g \]
holds. This concludes the proof.
\end{proof}

\end{rema}

\subsection{W-types in sheaves}

In this final subsection, we show that the axiom {\bf (WE)} is inherited by sheaf models. It turns out that the construction of W-types in categories of sheaves is considerably more involved than in the presheaf case (in \cite{bergmoerdijk08b} we showed that some of the complications can be avoided if the metatheory includes the axiom of choice). We then go on to show that the axiom {\bf (WS)} is inherited as well, if we assume the axiom of multiple choice. 

\begin{rema}{errorinMP}
In \cite{moerdijkpalmgren00} the authors claimed that W-types in categories of sheaves are computed as in presheaves (Proposition 5.7 in \emph{loc.cit.}) and can therefore be described in the same (relatively easy) way. But, unfortunately, this claim is incorrect, as the following  counterexample shows. Let $F: 1 \to 1$ be the identity map on the terminal object. The W-type associated to $F$ is the initial object, which, in general, is different in categories of presheaves and sheaves.  (This was noticed by Peter Lumsdaine together with the first author.)
\end{rema}

We fix a small map $F: Y \to X$ of sheaves. If $x \in X(a)$ and $S$ is a covering sieve on $a$, then we put
\[ Y_x^S := \{ (f: b \to a \in S, y \in Y(b)) \, : \, F(y) = x \cdot f \}. \]
Observe that $Y_x^S$ is small and write $\psi$ for the obvious projection
\[  \psi: \sum_{(S, x)} Y^S_x \to X \times_{\ct{C}_0} {\rm Cov}. \]
Let $\Psi = P_{\psi} \circ \inhspower$ and let $\cal{V}$ be its initial algebra (see \reftheo{initialalginpredcatofsmallmaps}). Elements $v$ of ${\cal V}$ are therefore of the form ${\rm sup}_{(a, x, S)}t$ with $(a, x, S) \in X \times_{\ct{C}_0} {\rm Cov}$ and $t: Y^S_x \to \inhspower {\cal V}$. We will think of such an element $v$ as a labelled well-founded tree, with a root labelled with $(a, x, S)$. To this root is attached, for every $(f, y) \in Y^S_x$ and $w \in t(f, y)$, the tree $w$ with an edge labelled with $(f, y)$. To simplify the notation, we will denote by $v(f, y)$ the \emph{small} collection of all trees that are attached to the root of $v$ with an edge that has the label $(f, y)$. 

We now wish to define a presheaf structure on ${\cal V}$. We say that a tree $v \in {\cal V}$ is \emph{rooted} at an object $a$ in \ct{C}, if its root has a label whose first component is $a$. If $v = {\rm sup}_{(a, x, S)}t$ is rooted at $a$ and $f: b \to a$ is a map in \ct{C}, then we can define a tree $v \cdot f$ rooted at $b$, as follows:
\[ v \cdot f = {\rm sup}_{(b, x \cdot f, f^* S)} f^* t, \]
with
\[ (f^*t)(g, y) = t(fg, y). \]
This clearly gives ${\cal V}$ the structure of a presheaf. Note that 
\[ (v \cdot f)(g, y) = v(fg, y). \]

Next, we define by transfinite recursion a relation on ${\cal V}$:
\begin{center}
\begin{tabular}{lcp{8 cm}}
$v \sim v'$ & $\Leftrightarrow$ & if the root of $v$ is labelled with $(a, x, S)$ and the root of $v'$ with $(a', x', S')$, then $a = a'$, $x = x'$ and there is a covering sieve $R \subseteq S \cap S'$ such that for every $(f, y) \in Y_x^R$ we have $v(f, y) \sim v'(f, y)$. 
\end{tabular}
\end{center}
Here, the formula $v(f, y) \sim v'(f, y)$ is supposed to mean
\[ \forall m \in v(f, y), n \in v'(f, y)\, : \, m \sim n. \]
In general, we will write $M \sim N$ for small subobjects $M$ and $N$ of ${\cal V}$ to mean 
\[ \forall m \in M, n \in N \, : \, m \sim n. \]
In a similar vein, we will write for such a subobject $M$,
\[ M \cdot f = \{ m \cdot f \, : \, m \in M \}. \]
That the relation $\sim$ is indeed definable can be shown by the methods of \cite{berg05} or \cite{bergmoerdijk08}. By transfinite induction one can show that $\sim$ is symmetric and transitive, and compatible with the presheaf structure ($v \sim w \Rightarrow v \cdot f \sim w \cdot f$). 

Next, we define \emph{composability} and \emph{naturality} of trees (as we did in the presheaf case, see \reftheo{Winpresh}). 
\begin{itemize}
\item A tree $v \in {\cal V}$ whose root is labelled with $(a, x, S)$ is \emph{composable}, if for any $(f: b \to a, y) \in Y^S_x$ and $w \in  v(f, y)$,  the tree $w$ is rooted at $b$.
\item A tree $v \in {\cal V}$ whose root is labelled with $(a, x, S)$ is \emph{natural}, if it is composable and for any $(f: b \to a, y) \in Y_x^S$ and $g: c \to b$,
\[ v(f, y) \cdot g \sim v(fg, y \cdot g). \]
\end{itemize}
One can show that if $v$ is natural, and $v \sim w$, then also $w$ is natural; moreover, natural trees are stable under restriction. The same applies to the trees that are \emph{hereditarily} natural (i.e. not only are they themselves  natural, but the same is true for all their subtrees). 

We shall write ${\cal W}$ for the object consisting of those trees that are hereditarily natural. The relation $\sim$ defines an equivalence on ${\cal W}$, for if a tree $v = {\rm sup}_{(a, x, S)} t$ is natural, then for all $(f, y) \in Y^S_x$ one has $v(f, y) \cdot \id \sim v(f \cdot \id, y \cdot \id)$, that is, $v(f, y) \sim v(f, y)$, and therefore $v \sim v$. By induction one proves that the equivalence relation $\sim$ on ${\cal W}$ is bounded and hence a quotient exists. We denote it by $\overline{\cal W}$. It follows from what we have said that the quotient $\overline{\cal W}$ is a presheaf, but more is true: one can actually show that $\overline{\cal W}$ is a sheaf and, indeed, the W-type associated to $F$ in sheaves.

\begin{lemm}{Wbarseparated}
Let $w, w' \in {\cal W}$ be rooted at $a \in \ct{C}$. If $T$ is a sieve covering $a$ and $w \cdot f \sim w' \cdot f$ for all $f \in T$, then $w \sim w'$. In other words, $\overline{\cal W}$ is separated.
\end{lemm}
\begin{proof} If the label of the root of $w$ is of the form $(a, x, S)$ and that of $w'$ is of the form $(a, x', S')$, then $w \cdot f \sim w' \cdot f$ implies that $x \cdot f = x' \cdot f$ for all $f \in T$. As $X$ is separated, it follows that $x = x'$.

Consider
\begin{eqnarray*}
R & = & \{ \, g: b \to a  \in (S \cap S') \, : \, \forall (h, y) \in Y_{x \cdot g}^{M_b} \, [ \,  w(gh, y) \sim w'(gh, y)\, ] \, \}.
\end{eqnarray*}
$R$ is a sieve, and the statement of the lemma will follow once we have shown that it is covering.

Fix an element $f \in T$. That $w \cdot f \sim w' \cdot f$ holds means that there is a covering sieve $R_f \subseteq f^*S \cap f^*S'$ such that for every $(k, y) \in Y^{R_f}_{x \cdot f}$ we have $w(fk, y) = (w \cdot f)(k, y) \sim (w' \cdot f)(k, y) = w'(fk, y)$. In other words, $R_f \subseteq f^*R$. So $R$ is a covering sieve by local character.
\end{proof}

\begin{lemm}{Wbarsheaf}
$\overline{\cal W}$ is a sheaf.
\end{lemm}
\begin{proof} Let $S$ be a covering sieve on $a$ and suppose we have a compatible family of elements $(\overline{w}_f \in \overline{\cal W})_{f \in S}$. Using the collection axiom, we know that there must be a span
\begin{displaymath}
\begin{array}{ccccc} 
S & \gets &  J  & \to & {\cal W} \\
f_j & \mapsfrom & j & \mapsto & w_j
\end{array}
\end{displaymath}
with $J$ small and $[w_j] = \overline{w}_{f_j}$ for all $j \in J$. Every $w_j$ is of form ${\rm sup}_{(a_j, x_j, R_j)}t_j$. If $f_j = f_{j'}$, then $w_j \sim w_{j'}$, so $x_j = x_{j'}$. Thus the $x_j$ form a compatible family and, since $X$ is a sheaf, can be glued together to obtain an element $x \in X(a)$. We claim that the desired glueing is $[w]$, where $w = {\rm sup}_{(a, x, R)}t \in {\cal V}$ is defined by:
\begin{eqnarray*}
R & = & \{ f_jg \, : \, j \in J, g \in R_j \}, \\
t(h, y) & = & \bigcup_{j \in J} \{t_j(g, y) \, : \, f_jg = h \} 
\end{eqnarray*}
For this to make sense, we first need to show that $w \in {\cal W}$, i.e., that $w$ is hereditarily natural. In order to do this, we prove the following claim.

\noindent
{\bf Claim.} Assume we are given $(h, y) \in Y_x^R$, with $h = f_j g$ for some $j \in J$. Then
\[ w(h, y) \sim w_j(g, y). \]
\begin{proof} Since 
\[ w(h, y) = \bigcup_{j' \in J} \{ w_{j'}(g', y) \, : \, f_{j'}g'  = h \}, \]
it suffices to show that $w_j(g, y) \sim w_{j'}(g', y)$ if $h = f_{j'}g'$.

By compatibility of the family $(\overline{w}_f \in \overline{\cal W})_{f \in S}$ we know that $w_j \cdot g \sim w_{j'} \cdot g' \in {\cal W}(c)$. This means that there is a covering sieve $T \subseteq g^*R_j \cap (g')^*R_{j'}$ such that for all $(k, z) \in Y_{x \cdot h}^T$, we have $(w_j \cdot g)(k, z) \sim (w_{j'} \cdot g')(k, z)$. So if $k: d \to c \in T$, then  
\begin{eqnarray*}
w_j(g, y) \cdot k & \sim & w_j(gk, y \cdot k) \\
& = & (w_j \cdot g)(k, y \cdot k) \\
& \sim & (w_{j'} \cdot g')(k, y \cdot k) \\
& = & w_{j'}(g'k, y \cdot k) \\
& \sim & w_{j'}(g', y) \cdot k.
\end{eqnarray*}
Because $\overline{\cal W}$ is separated (as was shown in \reflemm{Wbarseparated}), it follows that $w_j(g, y) \sim w_{j'}(g', y)$. This proves the claim.
\end{proof}

Any subtree of $w$ is a subtree of some $w_j$ and therefore natural. Hence we only need to prove of $w$ itself that it is composable and natural. Direct inspection shows that the tree that we have constructed is composable. For verifying that $w$ is also natural, let $(h: c \to a, y) \in Y_x^R$ and $k: d \to c$. Since $h \in R$, there are $j \in J$ and $g \in R_j$ such that $h = f_jg$. Then
\[
w(h, y) \cdot k \sim  w_j(g, y) \cdot k \sim w_j(gk, y \cdot k) \sim w(hk, y \cdot k),
\]
by using naturality of $w_j$ and the claim (twice).

It remains to show that $[w]$ is a glueing of all the $\overline{w}_f$, i.e., that $w \cdot f_j \sim w_j$ for all $j \in J$. So let $j \in J$. First of all, $x \cdot f_j = x_j$, by construction. Secondly, for every $g: c \to b \in R_j = (R_j \cap f_j^*R)$ and $y \in Y(c)$ such that $F(y) = x \cdot f_jg$, we have
\[ (w \cdot f_j)(g, y) = w(f_jg, y) \sim w_j(g, y). \]
This completes the proof.
\end{proof}

\begin{lemm}{Wbaralgebra}
$\overline{\cal W}$ is a $P_F$-algebra.
\end{lemm}
\begin{proof} We have to describe a natural transformation $S: P_F \overline{\cal W} \to \overline{\cal W}$. An element of $P_F \overline{\cal W}(a)$ is a pair $(x, t)$ consisting of an element $x \in X(a)$ together with a natural transformation $G: Y_x^{M_a} \to \overline{\cal W}$. Using collection, there is a map
\diaglab{spanforWbar}{ Y^{M_a}_x \ar[r]^(.45)t & \inhspower {\cal W} }
such that $[w] = G(y, f)$, for all $(f, y) \in Y_x^{M_a}$ and $w \in t(f, y)$. We define $S_xG$ to be 
\[  [ \, {\rm sup}_{(a, x, M_a)}t \, ]. \]
One now needs to check that $w$ is hereditarily natural. And then another verification is needed to check that $[w]$ does not depend on the choice of the map in \refdiag{spanforWbar}. Finally, one needs to check the naturality of $S$. These verifications are all relatively straightforward and similar to some of the earlier calculations, and therefore we leave all of them to the reader.
\end{proof}

\begin{lemm}{Wbarinitial}
$\overline{\cal W}$ is the initial $P_F$-algebra.
\end{lemm}
\begin{proof}
We will show that $S: P_F\overline{\cal W} \to \overline{\cal W}$ is monic and that $\overline{\cal W}$ has no proper $P_F$-subalgebras; it will then follow from Theorem 26 of \cite{berg05} (or Theorem 6.13 in \cite{bergmoerdijk08}) that $\overline{\cal W}$ is the W-type of $F$.

We first show that $S$ is monic. So let $(x, G), (x', G') \in P_FX(a)$ be such that $S_x G = S_{x'} G' \in \overline{\cal W}$. It follows that $x = x'$ and that there is a covering sieve $S$ on $a$ such that for all $(h, y) \in Y_x^S$, we have $G(h, y) = G'(h,  y)$. We need to show that $G = G'$, so let $(f, y) \in Y^{M_a}_x$ be arbitrary. For every $g \in f^*S$, we have:
\[ G(f, y) \cdot g  =  G(fg, y \cdot g) = G'(fg, y \cdot g) =  G'(f, y) \cdot g. \]
Since $f^*S$ is covering, it follows that $G(f, y) = G'(f, y)$, as desired.

The fact that $\overline{\cal W}$ has no proper $P_F$-subalgebras is a consequence of the inductive properties of ${\cal V}$ (recall that ${\cal V}$ is an initial algebra). Let ${\cal A}$ be a sheaf and $P_F$-subalgebra of $\overline{\cal W}$. We claim that 
\begin{eqnarray*}
 {\cal B} & = & \{ v \in {\cal V} \, : \, \mbox{if } v \mbox{ is hereditarily natural, then } [v] \in {\cal A} \}
\end{eqnarray*}
is a subalgebra of ${\cal V}$. Proof: Suppose $v$ is a tree that is hereditarily natural. Assume moreover that $v = {\rm sup}_{(a, x, S)}t$ and for all $(f, y) \in Y_x^S$ and $w \in t(f, y)$, we know that $[w] \in {\cal A}$. Our aim is to show that $[v] \in {\cal A}$.

For the moment fix an element $f: b \to a \in S$. Since $v \cdot f$ has a root labelled by $(b, x \cdot f, M_b)$ and $(v \cdot f)(g, y) = v(fg, y)$ for all $(g, y) \in Y^{M_b}_{x \cdot f}$, we have that $[v] \cdot f  = S_{x \cdot f} G$, where $G(g, y) = [v(fg, y)] \in {\cal A}$. Because ${\cal A}$ is a $P_F$-subalgebra of $\overline{\cal W}$ this implies that $[v] \cdot f \in {\cal A}$. Since this holds for every $f \in S$, while $S$ is a covering sieve and ${\cal A}$ is a subsheaf of $\overline{\cal W}$, we obtain that $[v] \in {\cal A}$, as desired.

We conclude that ${\cal B} = {\cal V}$ and hence ${\cal A} = \overline{\cal W}$. This completes the proof.
\end{proof}

To wrap up:
\begin{theo}{Wtypesstableundersheaves}
The axiom {\bf (WE)} is inherited by sheaf models.
\end{theo}

We believe that one has to make additional assumptions on ones predicative category with small maps $(\ct{E}, \smallmap{S})$ to show that the axiom {\bf (WS)} is inherited by sheaf models (the argument above does not establish this, the problem being that the initial algebra ${\cal V}$ will be large, even when the codomain of the map $F: Y \to X$ we have computed the W-type of is small). We will now show that this problem can be circumvented if we assume that the axiom of multiple choice {\bf (AMC)} holds in \ct{E}. It is quite likely that one can also solve this problem by using Aczel's Regular Extension Axiom: it implies the axiom {\bf (WS)} and is claimed to be stable under sheaf extensions (but, as far as we are aware, no proof of that claim has been published). 

\begin{theo}{stabilityofAMCundersheaves}
The axiom {\bf (AMC)} is inherited by sheaf models.
\end{theo}
\begin{proof}
This was proved in Section 10 of \cite{moerdijkpalmgren02}.
\end{proof}

\begin{theo}{stabilityofWSundersheaves}
(Assuming that {\bf (AMC)} holds in \ct{E}.) The axiom {\bf (WS)} is inherited by sheaf models.
\end{theo}
\begin{proof} We will continue to use the notation from the proof of the previous theorem. So, again, we assume we have a small map $F: Y \to X$ of sheaves. Moreover, we let $\psi$ be the map in \ct{E} and $\Psi$ be the endofunctor on \ct{E} defined above, we let ${\cal V}$ be its initial algebra and $\sim$ be the symmetric and transitive relation we defined on ${\cal V}$, and $\overline{\cal W}$ the W-type associated to $F$, obtained by quotienting the hereditarily natural elements in ${\cal V}$ by $\sim$.

Assume that $X$ is a small sheaf. Since {\bf (AMC)} holds in \ct{E}, it is the case that, internally in $\ct{E}/\ct{C}_0$, the map $\psi$ fits into a covering square as shown
\diag{ D \ar[r]^(.35)q \ar[d]_g & \sum_{(S, x)} Y^S_x \ar[d]^{\psi} \\
C \ar[r]_(.35)p & X \times_{\ct{C}_0} {\rm Cov}, }
in which all objects and maps are small in $\ct{E}/\ct{C}_0$ and $(g, q)$ is a collection span over $X \times_{\ct{C}_0} {\rm Cov}$. The W-type ${\cal U} = W_g$ in $\ct{E}/\ct{C}_0$ is small in $\ct{E}/\ct{C}_0$, because we are assuming that {\bf (WS)} holds in $\ct{E}$ (and hence also in $\ct{E}/\ct{C}_0$). The idea is to use this to show that $\overline{\cal W}$ is small as well.

Every element $u = {\rm sup}_c s \in {\cal U}$ determines an element in $\varphi(u) \in {\cal V}$ as follows: first compute $p(c) = (a, x, S)$. Then let for every $(y, f) \in Y^S_x$ the element $t(y, f)$ be defined by
\[ t(y, f) =  \{ (\varphi\circ s)(d) \, : \, d \in q_c^{-1}(y, f) \}. \]
Then $\varphi(u) = {\rm sup}_{(a, x, S)} t$  (so this is an inductive definition). We claim that for every hereditarily natural tree $v \in {\cal W}$ there is an element $u \in {\cal U}$ such that $v \sim \varphi(u)$. The desired result follows readily from this claim.

We prove the claim by induction: so let $v = {\sup}_{(a, x, S)} t$ be a hereditarily natural element of ${\cal V}$ and assume the claim holds for all subtrees of $v$. Since all subtrees of $v$ are hereditarily natural as well, this means that for every $(y, f) \in Y^S_x$ and $w \in t(y, f)$ there is an element $u \in {\cal U}$ such that $\varphi(u) = w$. From the fact that $(g, q)$ is a collection span over $X \times_{\ct{C}_0} {\rm Cov}$,  it follows that there is a $c \in C$ with $p(c) = (a, x, S)$ together with two functions: first one picking for every $d \in D_c$ an element $r(d) \in t(y, f)$ (because $t(y, f)$ is non-empty) and a second one picking for every $d \in D$ an element $s(d) \in {\cal U}$ such that $\varphi(s(d)) \sim r(d)$. It is not hard to see that $v \sim \varphi({\rm sup}_c s)$, using that $v$ is natural and therefore all elements in $t(y, f)$ are equivalent to each other.
\end{proof}

This completes the proof of our main result, \reftheo{maintheorem}.

\section{Sheaf models of constructive set theory}

Our main result \reftheo{maintheorem} in combination with \reftheo{existenceofinitZFalgebra} yields the existence of sheaf models for {\bf CZF} and {\bf IZF} (see \refcoro{maincorollary}). For the sake of completeness and in order to allow a comparison with classical forcing, we describe this model in concrete terms. We will not present verifications of the correctness of our descriptions, because they could in principle be obtained by unwinding the existence proofs, and other descriptions which differ only slightly from what we present here can already be found in the literature. 

To construct the initial $\spower$-algebra in a category of internal presheaves over a predicative category with small maps $(\ct{E}, \smallmap{S})$, let ${\cal W}$ be the initial algebra  of the endofunctor $\Phi = P_{\rm cod} \circ \spower$ on \ct{E} (see \reftheo{initialalginpredcatofsmallmaps}). Elements of $w \in {\cal W}$ are therefore of the form ${\rm sup}_c t$, with $c \in \ct{C}_0$ and $t$ a function from $\{ f \in \ct{C}_1 \, : \, {\rm cod}(f) = c \}$ to $\spower {\cal W}$. We think of such an element $w$ as a well-founded tree, where the root is labelled with $c$ and for every $v \in t(f)$, the tree $v$ is connected to the root of $w$ with an edge labelled with $f$. The object ${\cal W}$ carries the structure of a presheaf, with ${\cal W}(c)$ consisting of trees whose root is labelled with $c$, and with a restriction operation defined by putting for any $w = {\rm sup}_c t$ and $f: d \to c$,
\[ w \cdot f = {\rm sup}_d \, t(f \circ -). \]

The initial $\spower$-algebra ${\cal V}$ in the category of presheaves is constructed from ${\cal W}$ by selecting those trees that are hereditarily composable and natural:
\begin{itemize}
\item A tree $w = {\rm sup}_c (t) \in {\cal W}$ is \emph{composable}, if for any $f: d \to c$ and $v \in  t(f)$,  the tree $v$ has a root labelled with $d$.
\item A tree $w = {\rm sup}_c(t) \in {\cal W}$ is \emph{natural}, if it is composable and for any $f: d \to c$, $g: e \to d$ and $v \in t(f)$, we have $v \cdot g \in t(fg)$.
\end{itemize}
The \spower-algebra structure, or, equivalently, the membership relation on ${\cal V}$, is given by the formula ($x, {\rm sup}_c t \in {\cal V}$)
\[ x \in {\rm sup}_c t \Longleftrightarrow x \in t(\id_c). \]
The easiest way to prove the correctness of the description we gave is by appealing to Theorem 1.1 from \cite{kouwenhovenvanoosten05} (or Theorem 7.3 from \cite{bergmoerdijk08}). This model was first presented in the paper \cite{gambino05} by Gambino, based on unpublished work by Dana Scott.

The initial \spower-algebra in categories of internal sheaves is obtained as a quotient of this object ${\cal V}$. Roughly speaking, we quotient by bisimulation in a way which reflects the semantics of a category of sheaves. More precisely, we take ${\cal V}$ as defined above and we write: ${\rm sup}_c t \sim {\rm sup}_c t'$ if for all $f: d \to c$ and $v \in t(f)$, the sieve
\[ \{ \, g: e \to d \, : \, \exists v' \in t'(fg) \, ( \, v \cdot g \sim v' \, ) \, \} \]
covers $d$ and for all $f': d \to c$ and $v' \in t'(f')$, the sieve
\[  \{ \, g: e \to d \, : \, \exists v \in t(f'g) \, ( \, v' \cdot g \sim v \, ) \, \} \]
covers $d$. On the quotient the membership relation is defined by:
\[ [v] \in [{\rm sup}_c t] \Longleftrightarrow \mbox{the sieve } \{ f: d \to c \, : \, \exists v' \in t(f) \, ( \, v \cdot g \sim v' \, ) \}  \mbox{ covers } c. \] 
To see that this is correct, one should verify that $\sim$ defines a bounded equivalence relation and the quotient is a sheaf. Then one proves that it is the initial \spower-algebra by appealing to Theorem 1.1 from \cite{kouwenhovenvanoosten05} (or Theorem 7.3 from \cite{bergmoerdijk08}). The reader who wishes to see more details, should consult \cite{streicher09}.

\begin{rema}{classicalforcing}
To see the analogy with classical forcing (as in \cite{kunen80}, for example), note that any poset $\mathbb{P}$ determines a site, by declaring that $S$ covers $p$ whenever $S$ is dense below $p$. In this case, the elements of $\cal{V}$ are a particular kind of \emph{names} (as they are traditionally called). One could regard composability and naturality as saturation properties of names (so that, in effect, we only consider nice, saturated names). It is not too hard to show that every name (in the usual sense) is equal in a forcing model to such a saturated name, so that in the case of classical {\bf ZF} the models that we have constructed are not different from standard forcing models.
\end{rema}

\bibliographystyle{plain} \bibliography{ast}

\end{document}